\documentclass[leqno,10pt]{article}

\usepackage{amsmath,amsfonts,amssymb,amsthm}
\usepackage{graphicx}
\usepackage{bbm}
\usepackage[normalem]{ulem}
\usepackage{hyperref}
\hypersetup{
    colorlinks=true,
    linkcolor=blue,
    filecolor=magenta,
    urlcolor=cyan,
    citecolor=blue
}
\newcommand{\beq}{\begin{equation}}
\newcommand{\eeq}{\end{equation}}

\newcommand{\bea}{\begin{eqnarray}}
\newcommand{\eea}{\end{eqnarray}}
\newcommand{\beas}{\begin{eqnarray*}}
\newcommand{\eeas}{\end{eqnarray*}}

\newtheorem{theorem}{Theorem}[section]

\newtheorem{proposition}[theorem]{Proposition}

\newtheorem{lemma}[theorem]{Lemma}

\theoremstyle{definition}

\newtheorem{example}[theorem]{Example}

\newtheorem{remark}[theorem]{Remark}
\newtheorem{examples}[theorem]{Examples}
\newtheorem{foo}[theorem]{Remarks}

\newcommand{\R}{\mathbb R}
\newcommand{\Heis}{\mathbb{H}}
\newcommand{\ep}{\epsilon}
\newcommand{\p}{\mathbb P}

\newcommand{\bS}{\mathbb S}
\newcommand{\fH}{\mathbb{H}}
\newcommand{\E}{\mathbb E}

\newcommand{\h}{\mathcal H}
\newcommand{\cH}{\mathcal H}

\newcommand{\pO}{{\partial\Omega}}

\newcommand{\bH}{{\mathbf H}}
\newcommand{\bQ}{{\mathbf Q}}

\newcommand{\bi}{{\mathbf i}}

\DeclareMathOperator{\Vol}{Vol}
\DeclareMathOperator{\length}{length}
\DeclareMathOperator{\diver}{div}

\title{Heat content and horizontal mean curvature on the Heisenberg group}
\author{Jeremy Tyson\footnote{Supported by NSF Grant DMS-1600650. \hfill\break {\it Key words and phrases:} Heisenberg group, heat content, horizontal perimeter, horizontal mean curvature, Brownian motion.} \hspace{15pt} Jing Wang \\ \\ Department of Mathematics \\ University of Illinois \\ 1409 West Green St. \\ Urbana, IL, USA \\ \\ tyson@illinois.edu \hspace{15pt} wangjing@illinois.edu}
\date{\today}

\begin{document}
\maketitle

\begin{abstract}
We identify the short time asymptotics of the sub-Riemannian heat content for a smoothly bounded domain in the first Heisenberg group. Our asymptotic formula generalizes prior work by van den Berg--Le Gall and van den Berg--Gilkey to the sub-Riemannian context, and identifies the first few coefficients in the sub-Riemannian heat content in terms of the horizontal perimeter and the total horizontal mean curvature of the boundary. The proof is probabilistic, and relies on a characterization of the heat content in terms of Brownian motion.
\end{abstract}

\tableofcontents

\section{Introduction}
Let us begin by recalling the classical heat content problem in Euclidean space. Let $\Omega \subset \R^n$ be a bounded domain with finite volume $\Vol(\Omega)$ and finite perimeter $P(\Omega)$. Denote by $v(x,t)$ the solution to the heat equation in $\Omega$ with Dirichlet boundary condition:
\begin{equation*}
\begin{matrix} v_t = \tfrac12 \triangle v & \mbox{in $\Omega \times (0,\infty)$,} \\ v(x,t) = 0 & \mbox{for $(x,t) \in \pO \times(0,\infty)$}, \\ v(x,0) = 1 & \mbox{for $x \in \Omega$}. \end{matrix}
\end{equation*}
The {\bf heat content} of $\Omega$ at time $t>0$ is defined to be
$$
\bQ_\Omega(t) = \int_\Omega v(x,t) \, dx.
$$
The short time asymptotics of $\bQ_\Omega$ are controlled by geometric data involving the domain $\Omega$ and its boundary. Intuitively, one expects that the rate of escape of heat from $\Omega$ will depend, to first order, on the perimeter of $\Omega$. Moreover, it is reasonable to further conjecture that subsequent corrections should involve some type of curvature invariant of $\pO$. The following result of van den Berg and Le Gall \cite{vdblg:heat} formalizes this intuition. If $\Omega$ has $C^3$ smooth boundary, then
\begin{equation}\label{Q-Euclidean}
\bQ_\Omega(t) = \Vol(\Omega) - \sqrt{\frac{2t}{\pi}} \sigma(\pO) + \frac14 t \int_{\pO} H_\pO \, d\sigma + o(t),
\end{equation}
where $\sigma$ denotes the surface area measure on $\pO$ and $H_\Sigma(x)$ denotes the mean curvature of a surface $\Sigma$ at $x$. (For smoothly bounded domains, the surface area $\sigma(\pO)$ coincides with the perimeter $P(\Omega)$.)

The asymptotic expansion \eqref{Q-Euclidean} is closely related to Ledoux's characterization of perimeter in terms of the heat equation, inspired by de Giorgi's original definition of perimeter \cite{degiorgi:teoria}. Let
$$
p_t(x,y) = (2\pi t)^{-n/2} \exp(|x-y|^2/2t)
$$
and let $u(x,t) = \int_\Omega p_t(x,y) \, dy$ solve the heat equation $u_t = \tfrac12 \triangle u$ in $\R^n \times (0,\infty)$ with $u(x,0) = \mathbbm{1}_\Omega(x)$. The {\bf heat content of $\Omega$ in $\R^n$} at time $t>0$ is
\begin{equation}\label{H-Euclidean}
\bH_\Omega(t) = \int_\Omega u(x,t) \, dx = \iint_{\Omega \times \Omega} p_t(x,y) \, dy \, dx\,.
\end{equation}
Ledoux \cite{ledoux:semigroup} identified the perimeter of $\Omega$ as follows:
\begin{equation}\label{Ledoux}
P(\Omega) = \lim_{t\to 0} \sqrt{\frac{2\pi}{t}} \iint_{\Omega \times \Omega^c} p_t(x,y) \, dy \, dx\,.
\end{equation}
From \eqref{H-Euclidean} and \eqref{Ledoux} it is easy to see that
$$
\bH_\Omega(t) = \Vol(\Omega) - \sqrt{\frac{t}{2\pi}} P(\Omega) + o(\sqrt{t}).
$$
An instructive comparison of these two problems can be found in van den Berg \cite{vdb:heat-and-perimeter}, where also the situation for domains with nonsmooth boundary is considered. For smooth boundaries, higher order terms in the short time expansion of $\bH_\Omega(t)$ were obtained by Angiuli--Massari--Miranda \cite{amm:heat-content}. Van den Berg and Gilkey \cite{vdbg:manifold} extended the theory to Riemannian manifolds and obtained further terms in the short time expansion of $\bQ_\Omega(t)$. We refer the interested reader to a pair of excellent survey articles by Gilkey \cite{gil:survey1}, \cite{gil:survey2}.

The sub-Riemannian Heisenberg group and more general nilpotent stratified Lie groups (i.e., Carnot groups) provide a natural testing ground for analysis and geometry beyond the Riemannian setting. The connection between horizontal perimeter and the sub-Riemannian heat equation has already been studied by Bramanti--Miranda--Pallara \cite{bmp:bv}, who obtained a precise analog of Ledoux's characterization in step two Carnot groups. Recently, Marola--Miranda--Shanmugalingam \cite{mms:mms} generalized such results even further into the category of metric measure spaces supporting a Poincar\'e inequality. However, it appears that, up to now, more precise asymptotics for heat content have not been studied, even in the setting of the Heisenberg group.

In this paper we identify short time asymptotics for the sub-Riemannian heat content $\bQ_\Omega$ (in the sense of van den Berg and Le Gall) for a smoothly bounded domain in the first Heisenberg group.

Let $\Heis$ denote the first Heisenberg group, let $X_1$ and $X_2$ denote the standard frame for the horizontal distribution, and let $\triangle_0 = X_1^2+X_2^2$ denote the subelliptic Laplacian. See section \ref{sec:prelim1} for definitions. We fix a bounded domain $\Omega \subset \Heis$ with boundary $\pO$, and let $v(x,t)$ denote the solution to the heat equation with Dirichlet boundary conditions:
\begin{equation}\label{eq:Dirichlet}
\begin{matrix} v_t = \tfrac12\triangle_0 v & \mbox{in $\Omega \times (0,\infty)$,} \\ v(x,t) = 0 & \mbox{for $(x,t) \in \pO \times(0,\infty)$}, \\ v(x,0) = 1 & \mbox{for $x \in \Omega$}. \end{matrix}
\end{equation}
As before, the heat content of $\Omega$ at time $t$ is defined to be
$$
\bQ_{\Omega}(t) = \int_\Omega v(x,t) \, dx,
$$
where the integral is taken with respect to the Haar measure on $\Heis$ (which agrees with Lebesgue measure in $\R^3$), and we are interested in the short time asymptotics of $\bQ_\Omega$.

We denote by $\sigma_0$ the horizontal perimeter measure on $\pO$, which is defined provided $\pO$ is at least $C^1$, and by $H_{\pO,0}(x)$ the horizontal mean curvature at $x \in \pO$, which is defined provided $\pO$ is at least $C^2$. For definitions of and further discussion about these geometric quantities, see subsection \ref{subsec:perim-and-mean-curvature}. We remark that the horizontal mean curvature of a surface $\Sigma$ is only defined pointwise at noncharacteristic points. In this paper we will assume that the boundary of $\Omega$ has no characteristic points.

Our main theorem provides an exact analog of \eqref{Q-Euclidean} in the Heisenberg setting.

\begin{theorem}\label{thm-heat-content-prob}
Let $\Omega$ be a bounded domain in $\fH$ with boundary $\partial \Omega$ which is of class $C^3$ and which is completely noncharacteristic. Then the asymptotic expansion
\begin{equation}\label{Q-Heisenberg}
\bQ_{\Omega}(t) = \Vol(\Omega) - \sqrt{\frac{2t}{\pi}} \sigma_0(\pO) + \frac{t}{4} \int_{\pO} H_{\pO,0}(s) \, d\sigma_0(s) + o(t)
\end{equation}
holds in the limit as $t \to 0$.
\end{theorem}

This paper is structured as follows. Section \ref{sec:prelim1} reviews background material on the geometry of the Heisenberg group $\Heis$, especially the structure of tubular neighborhoods of smooth surfaces. Many results which we state are taken from a recent paper by Ritor\'e \cite{rit:tubular}. Section \ref{sec:prelim2} contains the necessary probabilistic preliminaries. We reformulate the problem in terms of the exit time of a Brownian motion process on $\Heis$, and perform a series of reductions which eventually allow us to deduce Theorem \ref{thm-heat-content-prob} from a corresponding theorem (Theorem \ref{thm-heat-content-prime}) for a stochastic process involving L\'evy's area form. We reduce the proof of the latter statement to three lemmas. In section \ref{sec:proof} we give the (rather technical) proofs of these lemmas. Some auxiliary calculations are deferred to an appendix for ease of exposition.

We conclude this introduction with some additional comments on the heat content problem in the Heisenberg group, and directions for future work.

First, we point out that there is an alternative approach to our main theorem which relies on the appearance of the sub-Riemannian metric on the Heisenberg group as a Gromov--Hausdorff limit of a sequence of Riemannian metrics. The use of this technique to establish results in sub-Riemannian geometry is by now a standard approach which has been used successfully by many authors. As a tool for understanding the sub-Riemannian geometry of submanifolds of Heisenberg groups, this approach featured prominently in the book \cite{cdpt:survey}. We anticipate that a careful analysis of the behavior of asymptotic formulas such as \eqref{Q-Euclidean} (or, more precisely, their Riemannian analogs as found in \cite{vdbg:manifold}) under degenerating limits of Riemannian metrics should reproduce our main asymptotic estimate \eqref{Q-Heisenberg} and possibly yield further terms in such expansions, similar to those found in Steiner's formula for the Carnot--Carath\'eodory metric \cite{BFFVW} and \cite{BTV}. We plan to return to this idea in a future paper.

The heat content $\bH_\Omega(t)$ of a domain $\Omega$ relative to the full Heisenberg group also deserves further study. As previously mentioned, the first order term (involving perimeter) in the short time expansion of $\bH_\Omega(t)$ has been identified by Bramanti, Miranda and Pallara, but analogs of the higher order formulas of Angiuli--Massari--Miranda \cite{amm:heat-content} remain unexplored in the Heisenberg setting, as do extensions to other Carnot groups. The case of higher dimensional Heisenberg groups, or perhaps general step two Carnot groups, should be a natural first step. Adapting the methods of this paper to those settings would require a precise understanding of the structure of tubular neighborhoods of hypersurfaces which is currently unavailable. The Riemannian approximation metholodogy described in the preceding paragraph, however, would in principle be effective in all such settings.

Finally, we would like to point out another possible extension of this heat content problem to other curved sub-Riemannian model spaces, such as the Cauchy-Riemann sphere $\bS^{2n+1}$ and anti-de Sitter space $AdS^{2n+1}$. Subelliptic heat kernels on these spaces are well understood, and explicit expressions can be obtained (see \cite{BB}, \cite{B}, \cite{CRS}, \cite{CRH}). In \cite{BW}, the authors studied Brownian motion processes on these model spaces as horizontal lifts of Brownian motions on complex protective space $\mathbb{CP}^n$ and complex hyperbolic space $\mathbb{CH}^n$ respectively, where the fiber motions are exactly given by the stochastic area processes on $\mathbb{CP}^n$ and $\mathbb{CH}^n$. Following a similar intuition as in present paper (as well as the analytic approach previously mentioned), one may proceed to obtain small time expansions of heat contents on these curved spaces, and observe the appearance of the curvatures of the ambient spaces.

\section{Geometric preliminaries}\label{sec:prelim1}

We model the Heisenberg group $\fH$ as the space $\R^3$ with the following group law:
\[
(x_1,x_2,x_3)*(y_1,y_2,y_3) = (x_1+y_1,x_2+y_2,x_3+y_3+x_1y_2-x_2y_1) \,.
\]
The left invariant vector fields
\[
X_1=\frac{\partial}{\partial x_1}-x_2\frac{\partial}{\partial x_3},\quad X_2=\frac{\partial}{\partial x_2}+x_1\frac{\partial}{\partial x_3},\quad X_3 =\frac{\partial}{\partial x_3}
\]
provide a global frame for the tangent bundle. The vector fields $X_1$ and $X_2$ span, at each point $x \in \fH$, the {\bf horizontal tangent space} $\cH_x\fH$, and an absolutely continuous curve $\gamma$ valued in $\Heis$ is said to be {\bf horizontal} if its tangent vector $\gamma'(t)$ lies in $\cH_{\gamma(t)}\Heis$ whenever it is defined. Introduce a metric $g_0$ on $\cH\fH$ by declaring $X_1$ and $X_2$ to be an orthonormal frame. The {\bf Carnot--Carath\'eodory (CC) metric} $d_{cc}$ is defined by
\begin{equation}\label{eq-cc-dist}
d_{cc}(x,y) = \inf \length_{cc}(\gamma)
\end{equation}
where the infimum is taken over all horizontal curves $\gamma:[a,b] \to \fH$ joining $x$ to $y$ and
\[
\length_{cc}(\gamma) = \int_a^b g_0(\gamma'(s),\gamma'(s))^{1/2}_{\gamma(s)} \, ds.
\]
The metric $d_{cc}$ is left invariant and geodesic. Explicit formulas for the CC geodesics will appear in subsection \ref{subsec:tubular-neighborhoods}. For later purposes we also introduce the Riemannian metric $g_1$ for which $X_1$, $X_2$ and $X_3$ are an orthonormal frame. Note that the lengths of any horizontal curve in the $g_0$ and $g_1$ metrics coincide. The ball with center $x$ and radius $r>0$ in the CC metric will be denoted $B_{cc}(x,r)$.

\subsection{Perimeter and mean curvature in the Heisenberg group}\label{subsec:perim-and-mean-curvature}

Let $\Omega$ be a bounded domain in $\Heis$ with $C^1$ boundary. For any $s\in\pO$, we consider the tangent space $T_s(\pO)$ at $s$ that is spanned by the vectors tangent to $\pO$. We say that $s$ is a {\bf characteristic point} if $T_s(\pO)$ agrees with the horizontal space $\h_s\Heis$, otherwise $s$ is said to be a {\bf non-characteristic point}. Throughout this paper, we assume that $\pO$ contains no characteristic points. Such an assumption, while clearly restrictive, nevertheless allows for a number of examples. For instance, there are smoothly bounded noncharacteristic tori in $\Heis$, see for example \cite[Remark 6.4]{tys:gcdima-heisenberg}.

Let $\sigma$ be the surface area measure on $\pO$, and let $\vec{n}(s)$ be the outward pointing unit $g_1$-normal at $s \in \pO$.
Let $\vec{n}_h$ be the orthogonal projection of $\vec{n}$ into $\h_s\Heis$; note that $\vec{n}_h \ne 0$ if and only if $s$ is noncharacteristic.
The {\bf horizontal perimeter measure} $\sigma_0$ on $\pO$ is $d\sigma_0 = |\vec{n}_h| \, d\sigma$.
We denote by $N(s)$ the normalized projection of the inward unit $g_1$-normal at $s$, i.e.\ $N(s) = -(\vec{n}_h/|\vec{n}_h|)(s)$.
If $\vec{n} = n_1 X_1 + n_2 X_2 + n_3 X_3$ then
\begin{equation}\label{eq:N}
N(s) = \frac{-n_1X_1-n_2X_2}{|(n_1,n_2)|}.
\end{equation}

Since $s$ is noncharacteristic, the space $T_s(\pO)\cap\h_s\Heis$ is one-dimensional. We call it the {\bf horizontal tangent space} $\h T_s(\pO)$ of $\pO$ at $s$, and we denote by $T(s)$ a unit vector which spans $\h T_s(\pO)$. Specifically, if $N(s)$ is as in the previous paragraph then we choose
\begin{equation}\label{eq:T}
T(s) = \frac{-n_2X_1+n_1X_2}{|(n_1,n_2)|}.
\end{equation}
The pair
\begin{equation}\label{eq:TN}
\{N(s),T(s)\}
\end{equation}
forms an orthonormal basis of $\h_s\Heis$ with respect to the sub-Riemannian metric $g_0$.

The horizontal tangent vector field $T$ generates a foliation of $\pO$, the {\bf Legendrian foliation}. If $\alpha$ is a curve in the Legendrian foliation with $\alpha(0) = s \in \pO$, then $\alpha'(0) = T(s)$.

Assuming that $\pO$ is $C^2$, the {\bf horizontal mean curvature} of $\pO$ at a point $s$ is defined as the horizontal divergence of the horizontal unit normal:
$$
H_{\pO,0}(s) = \diver_H\left(\frac{\vec{n}_h(s)}{|\vec{n}_h(s)|}\right)
$$
where $\diver_H(aX_1+bX_2) = X_1(a)+X_2(b)$. It is known (see e.g.\ \cite[Proposition 4.24]{cdpt:survey}) that $H_{\pO,0}(s)$ coincides with the planar curvature of the projection of the Legendrian curve $\alpha$ in $\pO$ through $s$ into the $x_1x_2$-plane.

\subsection{Tubular neighborhoods of the boundary of a smooth domain}\label{subsec:tubular-neighborhoods}

Since we only care about the heat loss within a small time---which can be felt close to the boundary $\pO$---it is natural to consider a small inner tubular neighborhood of $\pO$. For $\ep>0$ define
\[
\Omega_\ep=\lbrace x\in \Omega \,|\, \min_{y\in \fH\setminus \Omega} d_{cc}(x,y)<\ep \rbrace \,.
\]
We describe the structure of such tubular neighborhoods in a sequence of geometric lemmas. A detailed discussion is in the recent preprint by Ritor\'e \cite{rit:tubular}, where proofs of several of these lemmas can be found.

\begin{lemma}\label{lemma-unique-p}
Assume that $\pO$ is compact and smooth, without characteristic points. Then there exists $\ep>0$ sufficiently small so that for any $x\in\Omega_\ep$, there exists a unique point $s\in\pO$ which is nearest to $x$ in the CC-metric. Furthermore, $x$ is joined to $s$ by a unique CC geodesic.
\end{lemma}

\begin{proof}
Let $\mathrm{Unp(E)}$ be the set of points $x\in\fH$ for which there is a unique point of $E$ nearest to $x$.
For $s\in E$, define $\mathrm{reach}(E,s)$ as the supremum of those values $r>0$ for which $B(s,r)\subset\mathrm{Unp}(E)$. Let
\[
\mathrm{reach}(E):=\inf\{\mathrm{reach}(E,s)\,|\,s\in E\}.
\]
Then we just need to show that $\mathrm{reach}(\pO)>0$. This is proved in \cite[Theorem 4.2 and Theorem 4.5]{rit:tubular}. The uniqueness of the CC geodesic between $x$ and $s$ follows from \cite[Remark 3.5 and Section 4]{rit:tubular}.
\end{proof}

Let $x \in \Heis$. Each CC geodesic emanating from $x$ is contained in a maximal CC geodesic $\gamma_{x,v}^\lambda$ for some $v \in \h_x \Heis$, $|v|=1$, and some $\lambda \in \R$. Here $x = (\gamma_{x,v}^\lambda)(0)$ is the initial position, $v = (\gamma_{x,v}^\lambda)'(0)$ is the initial velocity vector and the parameter $\lambda$ is known as the {\bf curvature}. The explicit form of these geodesics is well known, cf.\ Section 2.2 in \cite{rit:tubular}.  If $v = \cos\theta X_1(x) + \sin\theta X_2(x)$ then
\begin{equation}\label{CCgeodesics}
\gamma_{x,v}^\lambda(t) = x*\left(\cos\theta\frac{\sin(\lambda t)}{\lambda} + \sin\theta \frac{1-\cos(\lambda t)}{\lambda}, -\cos\theta \frac{1-\cos(\lambda t)}{\lambda} + \sin\theta \frac{\sin(\lambda t)}{\lambda}, - \frac{\lambda t - \sin(\lambda t)}{\lambda^2}\right).
\end{equation}
The maximal CC geodesic $\gamma_{x,v}^\lambda$ is defined on the interval $(-2\pi/|\lambda|,2\pi/|\lambda|)$ (or on all of $\R$ if $\lambda = 0$). Its projection to the $x_1x_2$-plane is a circle of radius $1/|\lambda|$ if $\lambda \ne 0$, or is a line if $\lambda = 0$. The velocity vector at time $t$ is
$$
\dot\gamma_{x,v}^\lambda(t) = \cos(\theta - \lambda t)X_1(\gamma_{x,v}^\lambda(t)) + \sin(\theta - \lambda t)X_2(\gamma_{x,v}^\lambda(t)).
$$

The following lemma is Theorem 3.11 in \cite{rit:tubular}.

\begin{lemma}\label{lemma-h-normal-geodesic}
Assume $\pO$ is compact and $C^1$ smooth, without characteristic points. Fix $s\in \partial \Omega$ and $x\in\Omega$ such that $d_{cc}(x, \partial \Omega)=d_{cc}(x,s):=r$ and let $\gamma: [0,r]\to \fH$ be the minimizing CC-geodesic connecting $x$ and $s$ such that $\gamma(0)=x$ and $\gamma(r)=s$. Then
$-\dot{\gamma}(0)=N(s)$ where $N(s)$ is the inward horizontal unit normal vector to $\pO$ at $s$, see \eqref{eq:N}, and the curvature of $\gamma$ is $\lambda = 2g_1(\vec{n},X_3)/ |\vec{n}_h|$.

Moreover, for any $0\le t\le r$, $-\dot{\gamma}(t)=N(\gamma(t))$ is the inward horizontal unit normal vector to $\partial \Omega_t$, where
$$
\Omega_t = \{x\in \Omega \,|\, d_{cc}(x, \pO)<t\}\,.
$$
\end{lemma}

In view of the previous lemmas, we observe a foliated structure of $\Omega_\ep$ induced by the Carnot--Carath\'eodory distance to $\pO$. As in \eqref{eq:TN} we obtain a $g_1$-orthonormal frame $\{N,T,X_3\}$ defined along $\gamma$. We extend this to a smooth frame $\{N,T,Z\}$ defined in a neighborhood of $\gamma$.

\begin{lemma}
Let $\ep$ be as in Lemma \ref{lemma-unique-p}. For $x\in\Omega_\ep$, assume $d_{cc}(x, \partial \Omega)=r<\ep$, and let $\gamma$ be the unique geodesic connecting $x$ to $s \in \pO$ with $\gamma(0) = x$ and $\gamma(r) = s$. Then the frame $N,T,Z$ along the geodesic $\gamma$ admits a smooth extension as follows:
\begin{equation}\label{eq-N-T-ext}\begin{split}
&N(q)=-\bigg(\cos\theta+\lambda(q_2-x_2)\bigg)X_1(q)-\bigg(\sin\theta-\lambda(q_1-x_1)\bigg)X_2(q),\\
&T(q)=-\bigg(\sin\theta-\lambda(q_1-x_1)\bigg)X_1(q)+\bigg(\cos\theta+\lambda(q_2-x_2)\bigg)X_2(q),\\
&Z(q)=f(q)X_3(q),
\end{split}\end{equation}
where $\lambda$ is the curvature of $\gamma$ and
\begin{equation}\label{eq-N-T-f}
f(q)=(\cos\theta+\lambda(q_2-x_2))^2+(\sin\theta-\lambda(q_1-x_1))^2.
\end{equation}
Moreover,
\begin{align}\label{eq-N-T-brackets-1}
[N, T]=-2Z,\  [N,Z]=0,\ [T,Z]=-2\lambda Z
\end{align}
and, for $k$-fold iterated brackets,
\begin{equation}\label{eq-N-T-brackets-2}
[T, [T,\cdots[T,N]]]=(-1)^{k-1} 2^k \lambda^{k-1}Z, \quad [T, [T,\cdots[T,Z]]]=(-2\lambda)^{k}Z.
\end{equation}
\end{lemma}

\begin{proof}
The $g_1$-orthonormal frame $\{N,T,X_3\}$ along $\gamma$ is given by
$$
N(\gamma(t)) = -\cos(\theta - \lambda t) X_1(\gamma(t)) - \sin(\theta - \lambda t) X_2(\gamma(t))
$$
and
$$
T(\gamma(t)) = -\sin(\theta - \lambda t) X_1(\gamma(t)) + \cos(\theta - \lambda t) X_2(\gamma(t))
$$
The fact that the expressions in \eqref{eq-N-T-ext} define an extension of this frame follow from the formula \eqref{CCgeodesics} for the geodesic $\gamma = \gamma_{s,v}^\lambda$. Verification of the bracket identities \eqref{eq-N-T-brackets-1} and \eqref{eq-N-T-brackets-2} is a simple exercise, left to the reader.
\end{proof}

We next define a parametrization $\varphi_x$ of a neighborhood $\mathcal{O}_x$ of $\gamma$ by a neighborhood $D$ of the origin in $\mathbb{R}^3$. For $(\xi, y,z)\in \mathbb{R}^3$, we let
\begin{equation}\label{eq-cartesian}
\varphi_x(\xi,y,z) = \exp_x(-\xi N + y T + z Z),
\end{equation}
that is,
\begin{equation*}
\varphi_{x}(\xi,y,z):=c(1),
\end{equation*}
where $c(t) = (c_1(t),c_2(t),c_3(t))$ solves the differential equation
\begin{equation}\label{first-order-linear-system}
\dot{c}(t) = -\xi \, N(c(t)) + y \, T(c(t)) + z \, Z(c(t)), \qquad c(0) = x.
\end{equation}
We have introduced an additional minus sign in front of the coefficient of $N$ in \eqref{eq-cartesian} so that increasing values of the parameter variable correspond to motion from $x$ towards the boundary of $\Omega$; recall that $N$ is the inward pointing normal.

The first-order linear system \eqref{first-order-linear-system} can be solved explicitly. In Euclidean coordinates, it reads
\begin{align*}
& \dot{c}_1(t) = (\xi\cos\theta - y \sin\theta) + \lambda(\xi(c_2-x_2) + y(c_1-x_1)) \\
& \dot{c}_2(t) = (\xi\sin\theta + y \cos\theta) + \lambda(-\xi(c_1-x_1) + y(c_2-x_2)) \\
& \dot{c}_3(t) = z \, f(c(t)) + ((\xi\sin\theta + y \cos\theta) + \lambda(-\xi(c_1-x_1) + y(c_2-x_2)))c_1 \\
& \qquad \qquad - ((\xi\cos\theta - y \sin\theta) + \lambda(\xi(c_2-x_2) + y(c_1-x_1)))c_2 \,.
\end{align*}
The equations for $c_1(t)$ and $c_2(t)$ have solution
\begin{align*}
& c_1(t) = x_1 - \frac1\lambda(e^{\lambda t y}\cos(\lambda t \xi)-1)\sin\theta + \frac1\lambda e^{\lambda t y} \sin(\lambda t \xi) \, \cos\theta \\
& c_2(t) = x_2 + \frac1\lambda e^{\lambda t y} \sin(\lambda t \xi)\sin\theta + \frac1\lambda  (e^{\lambda t y}\cos(\lambda t \xi)-1) \, \cos\theta \,.
\end{align*}
Then $f(c(t)) = e^{2\lambda t y}$ and the equation for $c_3(t)$ has solution
\begin{align*}
& c_3(t) = x_3 + \frac1{2\lambda^2 y} (e^{2\lambda t y}-1) (-\xi + \lambda z) + \frac1{\lambda^2} e^{\lambda t y}\sin(\lambda t \xi) \\
& \qquad \quad + \frac1{\lambda} \left(x_1 \, (e^{\lambda t y} \cos(\theta - \lambda t \xi) - \cos\theta) + x_2 \, (e^{\lambda t y} \sin(\theta - \lambda t \xi) - \sin\theta ) \right).
\end{align*}
Hence
\begin{equation*}\begin{split}
\varphi_x(\xi,y,z)
&= \left( x_1 - \frac1\lambda ( e^{\lambda y} \sin(\theta - \lambda \xi) - \sin\theta), x_2 + \frac1\lambda (e^{\lambda y} \cos(\theta - \lambda \xi) - \cos\theta), \right. \\
& \qquad \quad x_3 + \frac1{2\lambda^2 y} (e^{2\lambda y}-1)(-\xi+\lambda z) + \frac1{\lambda^2} e^{\lambda y}\sin(\lambda \xi) \\
& \qquad \qquad + \frac1\lambda \left( x_1 (e^{\lambda y} \cos(\theta - \lambda \xi) - \cos\theta) + x_2 ( e^{\lambda y} \sin(\theta - \lambda \xi) - \sin\theta) \right) \,.
\end{split}\end{equation*}
The Jacobian of $\varphi_x$ is
$$
\det d\varphi_x(\xi,y,z) = \frac1{2\lambda y} e^{2\lambda y} ( e^{2\lambda y} - 1)
$$
which is always positive, hence $\varphi_x$ is locally invertible. Moreover, expressing the first two components of $\varphi_x$ in complex notation yields the map
$$
\xi + \bi y \mapsto (x_1+\bi x_2) + \frac1\lambda \bi e^{\bi \theta} \left( e^{-\bi \lambda (\xi + \bi y)} - 1 \right)
$$
which is invertible on a domain in $\mathbb{C}$ containing the interval $[0,r]$ if $r<\frac{2\pi}{|\lambda|}$. Hence $\varphi_x$ is invertible on a domain $D$ containing the interval $\{(\xi,0,0):0\le \xi\le r\}$, $\varphi_x(D) = \mathcal{O}_x$ is a domain in $\Heis$ containing the geodesic $\gamma$, $\varphi_x(0,0,0) = x$ and $\varphi_x(r,0,0) = s$. Using the group law we can verify that $x^{-1} * \varphi_x(\xi,y,z)$ is equal to
$$
\left( - \frac{e^{\lambda y}\sin(\theta-\lambda\xi) - \sin\theta}{\lambda} , \frac{e^{\lambda y}\cos(\theta-\lambda\xi)-\cos\theta}{\lambda},  \frac{e^{2\lambda y} - 1}{2\lambda^2 y}(-\xi + \lambda z) + \frac1{\lambda^2}e^{\lambda y}\sin(\lambda \xi) \right)\,.
$$
The inverse $\varphi^{-1}_x(\cdot)$ defines a Cartesian coordinate system in $\mathcal{O}_x$. Given $q\in \mathcal{O}_x$, we introduce the function
$$
\|\varphi^{-1}_x(q)\|:=\sqrt{|\xi(q)|^2+|y(q)|^2+|z(q)|}.
$$

\begin{lemma}\label{lemma-equi-dist}
The function $q \mapsto \|\varphi_x^{-1}(q)\|$ is comparable to the CC distance $d_{cc}(q,x)$ in the following sense: there exists a constant $K$ so that for all $q \in \mathcal{O}_x$, $K^{-1} ||\varphi^{-1}_x(q)|| \le d_{cc}(q,x) \le K ||\varphi^{-1}_x(q)||$.
\end{lemma}

\begin{proof}
The Kor\'anyi norm $|(y_1,y_2,y_3)|_H := ((|y_1|^2+|y_2|^2)^2+4y_3^2)^{1/4}$ defines a left invariant metric on $\Heis$ which is comparable to the CC metric. We will show that $|x^{-1}*q|_H$ is comparable to $||\varphi^{-1}_x(q)||$. It suffices to prove that
$$
(\xi^2+y^2)^2+4z^2
$$
is comparable to
\begin{equation}\label{comparison-quantity}\begin{split}
&\left( \left| \frac{e^{\lambda y}\sin(\theta-\lambda\xi) - \sin\theta}{\lambda} \right|^2 + \left| \frac{e^{\lambda y}\cos(\theta-\lambda\xi)-\cos\theta}{\lambda} \right|^2 \right)^2 \\
& \quad + 4 \left(  \frac{e^{2\lambda y} - 1}{2\lambda^2 y}(-\xi + \lambda z) + \frac1{\lambda^2}e^{\lambda y}\sin(\lambda \xi) \right)^2
\end{split}\end{equation}
when $(\xi,y,z)$ lies in a bounded region of $\R^3$. After some algebraic manipulation we rewrite \eqref{comparison-quantity} in the form
$$
\frac{4e^{2\lambda y}}{\lambda^4} \left( (\cosh(\lambda y)-\cos(\lambda\xi))^2 + \left( \frac{\sinh(\lambda y)}{y} (-\xi+\lambda z) + \sin(\lambda \xi) \right)^2 \right)\,.
$$
Let us denote the expression in the previous line by $G(\xi,y,z)$. The function $G$ is real analytic in all of $\R^3$. It is elementary but tedious to verify that
$$
\partial^{\alpha_1}_\xi \partial^{\alpha_2}_y \partial^{\alpha_3}_z G(0,0,0) = 0
$$
for all multi-indices $(\alpha_1,\alpha_2,\alpha_3)$ with $\alpha_1+\alpha_2+2\alpha_3 \le 3$, and $\partial_\xi^4 G(0,0,0) = \partial_y^4 G(0,0,0) = 24$, $\partial_\xi^2 \partial_y^2 G(0,0,0) = 8$, and $\partial_z^2 G(0,0,0)= 8$. By Taylor's theorem with remainder,
$$
G(\xi,y,z) = (\xi^4+2\xi^2 y^2 + y^4 + 4z^2)(1+o(1)) = ((\xi^2+y^2)^2+4z^4)(1+o(1))
$$
and so the desired comparison holds on bounded regions of $\R^3$.
\end{proof}

Throughout this paper, we often use the function $||\varphi_x^{-1}(\cdot)||$ in explicit computations.
The localized boundary $\varphi_{x}^{-1}(\partial \Omega\cap \mathcal{O}_x)$ has the representation
\begin{equation}\label{eq-NTZ-cor}
-\xi=h(y, z;s)-r,
\end{equation}
where $h(\cdot,\cdot;s):\R^2\to\R$ is smooth. Moreover, $h(y,z;s)$, for $s\in \pO$, satisfies the following expansion.

\begin{lemma}\label{lemma-h-H}
Let $\Omega$ and $\Omega_\ep$ be as in Lemma \ref{lemma-unique-p}. For $s\in \partial \Omega$ and $x\in\Omega$ such that $d_{cc}(x,\pO)=d_{cc}(x,s)=\ep$, let $\varphi_x$ be as in \eqref{eq-cartesian}. Then there exists $0<\delta<\ep$ such that for all $|(y,z)|<\delta$, it holds that
\begin{equation}\label{eq-h-H}
\bigg|h(y,z;s)-H_{\pO,0}(s)/2y^2-k_1(s)z\bigg|\le \delta^{-2}(|y|^3+|yz|),
\end{equation}
for some continuous function $k_1(\cdot)$ on $\pO\cap B_{cc}(s, \delta)$.
\end{lemma}

\begin{proof}
The parametrization $\varphi_x$ induces a diffeomorphism $d\varphi_x: \R^3\to T_s\fH$. In particular we have $d\varphi_x(\partial_y)=T(s)$. Since $T(s)\in T_s(\pO)$, we have $h_y(0,0;s)=0$ and $h_{yy}(0,0;s)=H_{\pO,0}(s)$. Hence \eqref{eq-h-H}  follows immediately from the Taylor expansion of $h(\cdot,\cdot;s)$ at $(0,0)$.
\end{proof}

The next lemma provides a way to change coordinates for integration. For a proof, see section 5 in \cite{rit:tubular}, specifically (5.7) and (5.8).

\begin{lemma}\label{lemma-J}
Let $\Omega$ and $\ep>0$ be as above. Consider the parametrization $\Psi$ of $\Omega_\ep$ by $\pO \times (0,\ep)$ given by $x = \Psi(s,r)$, where $r = d_{cc}(x,\pO) = d_{cc}(x,s)$. Equip $\pO \times (0,\ep)$ with the product of the horizontal perimeter measure $\sigma_0$ and Lebesgue measure, and equip $\Omega_\ep$ with the volume measure. Then the Jacobian $J_\Psi$ of $\Psi$ satisfies the estimate
\begin{align}\label{eq-jacobi}
|J_\Psi(s,r)-1+H_{\pO,0}(s)r|\le K_1 r^2
\end{align}
for all $s \in \pO$ and all $r \in (0,\ep)$, for some fixed constant $K_1>0$.
\end{lemma}

\begin{remark}
An explicit formula for the Jacobian $J_\Psi$ can be found in section 5 of \cite{rit:tubular}. For the purposes of our main result we only need the above first-order expansion in $r$.
\end{remark}

In the proof of Theorem \ref{thm-heat-content-prob} in subsection \ref{subsec:reduction2} we require information about the behavior of volume, horizontal perimeter, and total horizontal mean curvature for tubular neighborhoods and their boundaries in the $g_1$-metric.
The following lemma provides the necessary estimates. These estimates follow directly from the classical Steiner formula for volumes of tubular neighborhoods of submanifolds of Riemannian manifolds.

\begin{lemma}\label{lemma-compare-vol-peri-mean-curv}
Let $\Omega$ be a smoothly bounded domain in $\Heis$, and let $\Omega^r = \{x \in \Heis \, | \, d_{g_1}(x,\Omega) < r \}$ denote the $r$-neighborhood of $\Omega$ in the $g_1$-metric. Then
\begin{itemize}
\item[(1)] $\Vol(\Omega^r) = \Vol(\Omega) + O(r)$.
\item[(2)] $\int_{\pO^r} |\vec{n}_h| \, d\sigma = \int_{\pO} |\vec{n}_h| \, d\sigma + O(r)$.
\item[(3)] $\int_{\pO^r} H_{\pO^r,0} |\vec{n}_h| \, d\sigma = \int_{\pO} H_{\pO,0} |\vec{n}_h| \, d\sigma + O(r)$.
\end{itemize}
\end{lemma}

\begin{proof}
For sufficiently small $r>0$, the domain $\Omega^r \setminus \Omega$ is foliated by the surfaces $\pO^t$, $0<t<r$. Define functions $A$ and $B$ in $\Omega^r \setminus \Omega$ by $A(x) = |\vec{n}_h(x)|$ and $B(x) = H_{\pO^t,0}(x)$ for $x \in \pO^t$. Then $A$ and $B$ are smooth in $\Omega^r \setminus \Omega$.  Our starting point is the Steiner formula
$$
\Vol(\Omega^r) = \Vol(\Omega) + \int_0^r \sigma(\pO^t) \, dt,
$$
where $\sigma$ denotes the surface measure in the $g_1$ metric. For sufficiently small $r>0$, the domain $\Omega^r \setminus \Omega$ may be parameterized by $\pO \times (0,r)$ (analogously to the discussion in this section in the setting of the Carnot--Carath\'eodory metric) via a diffeomorphism $\Psi$, and for a smooth function $f:\Omega^r \setminus \Omega \to \R$,
$$
\int_{\pO^t} f(x) \, d\sigma(x) = \int_{\pO} f(\exp_s(t\vec{n}(s))) \, J_\Psi(s,t)\, d\sigma(s)\,, \qquad x = \Psi(s,t),
$$
cf.\ \cite[Lemma 3.12]{gray:tubes}. Expanding in a series in $t$ and using the analog of Lemma \ref{lemma-J} for the $g_1$ metric gives
$$
\int_{\pO^t} f \, d\sigma = \int_{\pO} (f + g_1(\nabla f,\vec{n})t + o(t)) (1+H_{\pO,1}\,t+o(t)) \, d\sigma\,
$$
where $H_{\pO,1}$ denotes the mean curvature in the $g_1$ metric. Thus
$$
\int_{\pO^t} f \, d\sigma = \int_{\pO} f \, d\sigma + \int_{\pO} \bigl( g_1(\nabla f,\vec{n}) + f\,H_{\pO,1}) \bigr) d\sigma \cdot t + o(t)
$$
Part (2) follows by choosing $f = A$ and part (3) by choosing $f = AB$, where $A$ and $B$ are as defined at the start of this proof. Finally, (1) follows from Steiner's formula above.
\end{proof}

\section{Probabilistic preliminaries}\label{sec:prelim2}

\subsection{First reduction: time change}\label{subsec:time-change}

The interpretation of the solution of a Dirichlet problem in terms of the exit time of the corresponding Markov process is well-known and has been widely used. Let ${x}_t$ be the strong Markov process generated by the horizontal sub-Laplacian $\frac12\triangle_0$ starting from $x\in\fH$. Then the solution $v(x,t)$ of the Dirichlet heat equation \eqref{eq:Dirichlet} yields the probability of surviving up to time $t$:
\[
v(x,t)=\p_x({T}_\Omega>t),
\]
where
\[
{T}_\Omega=\inf\{t>0, {x}_t\in \fH \setminus \Omega \}.
\]
Intuitively, the most likely event is that the Markov process escapes $\Omega$ in the direction of the outward horizontal normal $-N$ at the boundary $\pO$. It is more convenient for us to locally use the frame that is equipped with such information. For each $x\in\fH$, consider the new frame $\{N,T,Z\}$ as in \eqref{eq-N-T-ext}. In a small neighborhood $\mathcal{O}_x$, the horizontal sub-Laplacian can be written as
\begin{equation}\label{eq-delta-L}
\triangle_0 = \frac{1}{f}(N^2+T^2),
\end{equation}
where $f$ is as in \eqref{eq-N-T-f}. We write
$$
L=N^2+T^2.
$$
Let $\tilde{x}_t$ the Markov process generated by $L$ and starting from $x$. Then $\tilde{x}_t$ solves the Stratonovich differential equation
\begin{equation}\label{eq-SDE}
\begin{cases}
d\tilde{x}_t=-N(\tilde{x}_t)dB_t^N+T(\tilde{x}_t)dB_t^T \\
\tilde{x}_0=x
\end{cases}
\end{equation}
where $B_t^N$, $B_t^T$ are independent standard Brownian motions. By using the language of stochastic flows we can lift the process to the tangent space $T_x\fH$. Combining Strichartz's result (\cite{str:cbhd}, Theorem 3.2) with \eqref{eq-N-T-brackets-1} and \eqref{eq-N-T-brackets-2}, we deduce that
\begin{equation}\label{eq-sub-BM}
\tilde{x}_t=\exp_x\left( -B^N_tN+B^T_tT+ tR_tZ\right)
\end{equation}
where $R_t$ is a remainder term (process) which satisfies the following estimate: $\exists\, \alpha_0, c_0>0$ such that for any $R>c_0$,
\begin{equation}\label{eq-R-1-est}
\p\bigg(\sup_{0\le s\le t}|R_s|\ge R\bigg)\le \exp\bigg(-\frac{R^{\alpha_0}}{c_0t} \bigg)
\end{equation}
The derivation of \eqref{eq-R-1-est} is an easy consequence of the result of Azencott  \cite[p.\ 252]{aze:formule}, see also Castell \cite[p.\ 235]{cas:asymptotic}. Moreover, if we write
\[
\tilde{X}_t=\varphi_x^{-1}(\tilde{x}_t)=(-B^N_t, B^T_t, tR_t),
\]
then we have the following tail estimates. We remind the reader that $q \mapsto \|\varphi_x^{-1}(q)\|$ refers to the homogeneous distance considered in Lemma \ref{lemma-equi-dist}; this notation will be used repeatedly in what follows.

\begin{lemma}\label{lemma-X-s-est}
Let $\tilde{X}_t$ be given as above. Then the following estimates hold when $t$ is small enough.
\begin{itemize}
\item[(1)] For any $0<\alpha<1$, there exist $ c, C,  \alpha'>0$ such that
\begin{equation}\label{eq-X-s-est}
\p\bigg(\sup_{0\le s\le t}||\tilde{X}_s||^2> t^{1-\alpha}\bigg)\le C\exp\bigg(-\frac{c}{t^{\alpha'}} \bigg).
\end{equation}
\item[(2)] For any $\delta>0$, there exists $C>0$ such that
\begin{equation}\label{eq-X-TD-tilde-1}
 \p_x\left(\sup_{0\le s\le t}||\tilde{X}_s||\ge \delta\right)\le Ce^{-\frac{\delta^2}{16t}}.
 \end{equation}
\item[(3)] There exist  $c, c'>0$ such that
 \begin{equation}\label{eq-X-TD-tilde-2}
 \p_x(\tilde{T}_\Omega<t)\le c'e^{-\frac{d_{cc}^2(x,\pO)}{ct}},
 \end{equation}
 where $d_{cc}(x,\pO)$ is Carnot-Carath\'eodory distance from $x$ to $\partial \Omega$.
 \item[(4)] (Principle of not feeling the boundary)
Let $\Omega$ and $\Omega_\ep$ be as given previously, then
\begin{equation}\label{eq-not-feel-bdry-tilde}
\int_{\Omega\setminus \Omega_\ep}\p_x(\tilde{T}_\Omega>t)dx=\Vol(\Omega)- \Vol(\Omega_\ep)+ O(e^{-\ep^2/ct})
\end{equation}
for some constant $c>0$.
\end{itemize}
\end{lemma}

\begin{proof}
Note $||X_s||^2=|B^N_t|^2+|B^T_t|^2+|tR_t |$, and for $B^i_t$, $i=N, T$, we know that for any $\alpha>0$ there exists $c>0$ such that
\[
\p\bigg(\sup_{0\le s\le t}|B^i_s|^2> t^{1-\alpha}\bigg)=
\p\bigg(\sup_{0\le s\le 1}|B^i_s|^2> t^{-\alpha}\bigg)
\le \exp\bigg(-\frac{c}{t^{\alpha}} \bigg).
\]
Moreover, for $t\in[0,1)$ small enough, by \eqref{eq-R-1-est} we have
\[
\p\bigg(\sup_{0\le s\le t}|sR_s|> t^{-\alpha}\bigg)\le \p\bigg(\sup_{0\le s\le 1}|R_s|> t^{-1-\alpha}\bigg)
\le \exp\bigg(-\frac{1}{c_0 t^{\alpha_0\alpha}} \bigg).
\]
We then complete the proof of (1) by letting $\alpha'=\min\{ \alpha, \alpha_0\alpha\}$. The proof of (2) follows the same argument as that of (1).
To see (3), just note that
\[
\p_x(\tilde{T}_\Omega<t)\le \p_x\left( \sup_{0\le s\le t}d_{cc}(\tilde{x}_s, x)>d_{cc}(x,\pO)\right).
\]
Due to the equivalence between $d_{cc}(x,y)$ and $\|\varphi^{-1}_x(y)\|$ for any $y\in\mathcal{O}_x$, there exists  $C>0$ such that
\[
 \p_x\left( \sup_{0\le s\le t}d_{cc}(\tilde{x}_s, x)>d_{cc}(x,\pO)\right) \le \p_x\left( \sup_{0\le s\le t}||\tilde{X}_s||>Cd_{cc}(x,\pO)\right).
\]
By plugging $\delta=d_{cc}(x,\pO)$ into \eqref{eq-X-TD-tilde-1} we obtain \eqref{eq-X-TD-tilde-2}. At last, from (3) we have
\[
\int_{\Omega\setminus \Omega_\ep}\p_x(\tilde{T}_\Omega>t)dx=\Vol(\Omega\setminus \Omega_\ep)(1-O(e^{-\ep^2/ct})),
\]
which immediately implies \eqref{eq-not-feel-bdry-tilde}.
\end{proof}

Next, from \eqref{eq-delta-L} we know that ${x}_t$ is a time-changed version of $\tilde{x}_t$. Precisely, let $\mathfrak{t}(t)=\int_0^t{f(\tilde{x}_s)}ds$ and $\mathfrak{t}^{-1}(t)=\sup\{s:\mathfrak{t}(s)\le t \}$, then we have
\[
{x}_t=\tilde{x}_{\mathfrak{t}^{-1}(t)}.
\]
The exit time of $\tilde{x}_t$ is $\tilde{T}_\Omega=\mathfrak{t}^{-1}(T_\Omega)$, hence
\begin{equation}\label{eq-T-Op-Tilde-OP}
\p_x(\tilde{T}_\Omega>t)=\p_x\bigg({T}_\Omega>\int_0^{t}{f(\tilde{x}_u)}du \bigg).
\end{equation}
Denote by $\tilde{\bQ}_\Omega(t)$ the heat content associated with $\tilde{x}_t$. Then we can easily show that $\tilde{\bQ}_\Omega(t)$ differs from ${\bQ}_\Omega(t)$ by $o(t)$.

\begin{proposition}\label{prop-beta}
Let ${\bQ}_\Omega(t)$ and $\tilde{\bQ}_\Omega(t)$ be given as above. Then
\[
{\bQ}_\Omega(t)=\tilde{\bQ}_\Omega(t)+o(t).
\]
\end{proposition}

\begin{proof}
From \eqref{eq-N-T-f} we know that $1 - c \lambda \, d_{cc}(x,y) \le f(y) \le 1 + c \lambda \, d_{cc}(x,y)$, provided $y$ is sufficiently close to $x$. Moreover, from Lemma \ref{lemma-equi-dist} and \eqref{eq-X-s-est} we know that for any $0<\alpha<1$ there exist $c_1, \alpha', c', C>0$ such that
\[
\p_x\bigg( \sup_{0\le s\le t} d_{cc}(\tilde{x}_s,x)>t^{1-\alpha}\bigg)
\le \p \bigg( \sup_{0\le s\le t} ||\varphi_x^{-1}(\tilde{x}_s)||>c_1t^{1-\alpha}\bigg)
\le Ce^{-c'/t^{\alpha'}},
\]
hence
\begin{align*}
\p_x\bigg({T}_\Omega>\int_0^{t}{f(\tilde{x}_u)}du\bigg)&\ge\p_x\bigg({T}_\Omega>\int_0^{t}{(1+c\lambda t^{1-\alpha})}du\bigg)+O(e^{-c'/t^{\alpha'}})\\
&=\p_x\bigg({T}_\Omega>t+{c\lambda }t^{2-\alpha}\bigg)+O(e^{-c'/t^{\alpha'}})
\end{align*}
and
\begin{align*}
\p_x\bigg({T}_\Omega>\int_0^{t}{f(\tilde{x}_u)}du\bigg)&\le\p_x\bigg({T}_\Omega>\int_0^{t}{(1-c\lambda t^{1-\alpha})}du\bigg)+O(e^{-c'/t^{\alpha'}})\\
&=\p_x\bigg({T}_\Omega>t-{c\lambda }t^{2-\alpha}\bigg)+O(e^{-c'/t^{\alpha'}})
\end{align*}
Therefore by \eqref{eq-not-feel-bdry-tilde} and \eqref{eq-T-Op-Tilde-OP} we have
\begin{align*}
\int_{\Omega_\ep}\p_x(\tilde{T}_\Omega>t)dt
&\ge\int_{\Omega_\ep}\p_x\bigg({T}_\Omega>t-{c\lambda }t^{2-\alpha}\bigg)dt+O(e^{-c'/t^{\alpha'}})\\
&={\bQ}_\Omega\bigg(t-{c\lambda }t^{2-\alpha}\bigg)+O(e^{-c'/t^{\alpha'}}).
\end{align*}
and
\begin{align*}
\int_{\Omega_\ep}\p_x(\tilde{T}_\Omega>t)dt
&\le\int_{\Omega_\ep}\p_x\bigg({T}_\Omega>t+\frac{c\lambda }{2-\alpha}t^{2-\alpha}\bigg)dt+O(e^{-c'/t^{\alpha'}})\\
&={\bQ}_\Omega\bigg(t+{c\lambda }t^{2-\alpha}\bigg)+O(e^{-c'/t^{\alpha'}}).
\end{align*}
Observe that $\sqrt{t\pm{c\lambda }t^{2-\alpha}}=\sqrt{t}+o(t^{3/2-\alpha})$. Applying the `principle of not feeling the boundary' for $x_t$ we know that
$\tilde{\bQ}_\Omega(t)=\int_{\Omega_\ep}\p_x(\tilde{T}_\Omega>t)dt+O(e^{-\ep^2/t})={\bQ}_\Omega(t)+o(t)$. The proof is complete.
\end{proof}

\subsection{Second reduction: eliminating higher order remainder terms}\label{subsec:reduction2}

Let us denote L\'evy's area process by $A_t := \int_0^t (B_s^NdB_s^T-B_s^TdB_s^N)$, and consider the `truncated process'
\begin{equation}\label{eq-x-prime}
x_t':=\exp_x\left( -B^N_tN+B^T_tT+ A_tZ\right).
\end{equation}
By \cite[Theorem 2.1]{cas:asymptotic} we know that
\[
x_t'=\tilde{x}_t+t^{3/2}P_t
\]
where $P_t$  satisfies that $\exists\, \alpha_1, c_1>0$ such that for any $R>c_1$,
\[
\p\bigg(\sup_{0\le s\le t}|P_s|\ge R\bigg)\le \exp\bigg(-\frac{R^{\alpha_1}}{c_1t} \bigg),
\]
where $|\cdot|$ is the Euclidean norm in $\R^3$. Since the Riemannian metric $g_1$ and the Euclidean metric are locally bi-Lipschitz equivalent, we may equivalently write
\begin{equation}\label{eq-P-est}
\p\bigg(\sup_{0\le s\le t} s^{-3/2} d_{g_1}(x_s',x_s)\ge R\bigg)\le \exp\bigg(-\frac{R^{\alpha_1}}{c_1t} \bigg),
\end{equation}
for a possibly different choice of $c_1$. Consider the associated process on $T_x\fH$,
\begin{equation}\label{eq-X-NTZ}
X_t:=\varphi^{-1}_x(x_t')=\bigg(-B^N_t,\,B^T_t,\, A_t \bigg).
\end{equation}
Following the same arguments, we easily obtain that Lemma \ref{lemma-X-s-est} holds for $X_t$ as well.

\begin{lemma}\label{lemma-X-TD-prime}
Let ${X}_t$ be given as above. Then the following estimates hold when $t$ is small enough.
\begin{itemize}
\item[(1)] For any $0<\alpha<1$, there exist $C, c, \alpha'>0$ such that
\begin{equation}\label{eq-X-s-est-prime}
\p\bigg(\sup_{0\le s\le t}||{X}_s||^2> t^{1-\alpha}\bigg)\le C\exp\bigg(-\frac{c}{t^{\alpha'}} \bigg).
\end{equation}
\item[(2)] For any $\delta>0$, there exists $C>0$ such that
\begin{equation}\label{eq-X-TD-delta-1}
 \p_x\left(\sup_{0\le s\le t}||{X}_s||\ge \delta\right)\le Ce^{-\frac{\delta^2}{16t}}.
 \end{equation}
\item[(3)] There exist  $c, c'>0$ such that
 \begin{equation}\label{eq-X-TD-delta-2}
 \p_x({T}'_\Omega<t)\le c'e^{-\frac{d_{cc}^2(x,\pO)}{ct}},
 \end{equation}
 where $d_{cc}(x,\pO)$ is Carnot-Carath\'eodory distance from $x$ to $\partial \Omega$.
 \item[(4)] (Principle of not feeling the boundary)
Let $\Omega$ and $\Omega_\ep$ be as given previously, then
\begin{equation}\label{eq-not-feel-bdry-prime}
\int_{\Omega\setminus \Omega_\ep}\p_x({T}'_\Omega>t)dx=\Vol(\Omega)- \Vol(\Omega_\ep)+ O(e^{-\ep^2/ct})
\end{equation}
for some constant $c>0$, where $T'_{\Omega}=\inf\{t>0, {x}'_t\in \fH \setminus \Omega \}$.
\end{itemize}
\end{lemma}

Let  $\bQ'_\Omega(t)=\int_{\Omega}\p_x({T}'_\Omega>t)dt$. We then have the following heat content expansion for $\bQ'_\Omega(t)$ when $t\to0$.

\begin{theorem}\label{thm-heat-content-prime}
Let $\Omega \subset \Heis$ be a bounded domain in $\fH$ whose boundary is smooth and has no characteristic points. Let $x'_t$ be the process given in \eqref{eq-x-prime}. Then the associated heat content has the following expansion
\begin{equation}\label{heat-content-expansion}
{\bQ}'_{\Omega}(t) = \Vol(\Omega) - \sqrt{\frac{2t}{\pi}} \sigma_0({\partial \Omega}) + \frac{t}{4} \int_{\partial \Omega} H_{\pO,0}(s) \, d\sigma_0(s) + o(t)
\end{equation}
as $t \to 0$.
\end{theorem}
We postpone the proof of the above theorem to Subsection \ref{subsec:reduction3} and Section \ref{sec:proof}. In the rest of this section, we sketch the proof of the main theorem.

\begin{proof}[Proof of Theorem \ref{thm-heat-content-prob}]
From \eqref{eq-not-feel-bdry-prime} we know that $\bQ'_\Omega(t)=\int_{\Omega_\ep}\p_x({T}'_\Omega>t)dt+O(e^{-\ep^2/t})$.
Moreover, note
\begin{equation*}\begin{split}
\p_x(\tilde{T}_\Omega>t)
&= \p_x\bigg(\forall\, 0\le s\le t, \tilde{x}_s\in\Omega \bigg) \\
&\le \p_x\bigg(\forall\, 0\le s\le t, x'_s\in\Omega^+, \sup_{0\le s\le t} s^{-3/2}d_{g_1}(x'_s,\tilde{x}_s)<R \bigg)+O(e^{-C_R/t})
\end{split}\end{equation*}
for $\Omega^+=\{x\in\fH, d_{g_1}(x, \Omega)\le t^{3/2}R\} $. Here $C_R>0$ is a constant depending on $R$, and the last inequality comes from \eqref{eq-P-est}.
Also since
\begin{align*}
\p_x(\tilde{T}_\Omega>t)
&=\p_x\bigg(\forall\, 0\le s\le t, \tilde{x}_s\in\Omega \bigg)\\
&\ge\p_x\bigg(\forall\, 0\le s\le t, {x}'_s\in\Omega^-,  \sup_{0\le s\le t} s^{-3/2}d_{g_1}(x'_s,\tilde{x}_s)<R  \bigg)\\
&\ge \p_x\bigg(\forall\, 0\le s\le t, {x}'_s\in\Omega^-  \bigg)+O(e^{-C_R/t}),
\end{align*}
where $\Omega^-=\{x\in\fH, d_{g_1}(x, \Omega^c)\ge t^{3/2}R\}$, we have
\[
\p_x({T}'_{\Omega^-}>t)+o(t)\le \p_x(\tilde{T}_\Omega>t)\le \p_x({T}'_{\Omega^+}>t)+o(t).
\]
By the principle of not feeling the boundary we have ${\bQ}'_{\Omega^-}(t)+o(t)\le\tilde{\bQ}_\Omega(t)\le{\bQ}'_{\Omega^+}(t)+o(t)$.
Moreover, from Lemma \ref{lemma-compare-vol-peri-mean-curv} and Theorem \ref{thm-heat-content-prime} we obtain
\begin{equation}\label{eq-claim-beta}
{\bQ}'_\Omega(t)={\bQ}'_{\Omega^\pm}(t)+o(t).
\end{equation}
Therefore  $\tilde{\bQ}_\Omega(t)={\bQ}'_{\Omega}(t)+o(t)$. Together with Proposition \ref{prop-beta} we have
\[
{\bQ}_\Omega(t)={\bQ}'_{\Omega}(t)+o(t).
\]
Hence we complete the proof.
\end{proof}

\subsection{Third reduction: decomposing the main event into subevents}\label{subsec:reduction3}

In this section we reduce Theorem \ref{thm-heat-content-prime} to a sequence of lemmas. Following the intuition that the Markov process $x_t$ is most likely to exit $\Omega$ along the outward horizontal normal direction of the boundary, we track the furthest distance that $B^N$ can travel before time $t$ by considering the following process $\tau_t$. For each $t>0$,
\begin{equation}\label{eq-B^N-tau}
B^N_{\tau_t}=\sup_{0\le \tau\le t}B^N_{\tau}.
\end{equation}
The joint density of $B^N_{\tau_t}$ and $\tau_t$ is known.

\begin{lemma}\label{lemma-joint}
The joint density of $( B^N_{\tau_t}, \tau_t)$ is given by
\begin{equation}\label{eq-phi}
\Phi(\xi, \tau;t)=\frac{\xi e^{-\frac{\xi^2}{2\tau}}}{\pi\tau^{3/2}(t-\tau)^{1/2}}{\mathbbm{1}}_{[0,t)}(\tau){\mathbbm{1}}_{[0,\infty)}(\xi)
\end{equation}
\end{lemma}

For a proof, see \cite[p.\ 339]{louchard1968mouvement}.

Moreover, the event $\{x'_{\tau_t}\in \Omega\}$ captures the major part of the event that the process stays inside $\Omega$, namely $\{T'_\Omega>t\}$. We will estimate $\p_x(x'_{\tau_t}\in \Omega)$ as well as its difference from $\p_x(T'_\Omega>t)$.
Since
\[
 \p_x(x'_{\tau_t}\in \Omega)-\p_x(T'_\Omega>t)=\p(\tau_t<T'_\Omega\le t)+\p_x(T'_\Omega\le\tau_t\le t, x'_{\tau_t}\in \Omega).
 \]
we just need to estimate each of the terms
$$
\int_{\Omega_\ep} \p_x(x'_{\tau_t}\in \Omega)dx,
$$
$$
\int_{\Omega_\ep}\p_x(\tau_t<T'_\Omega\le t)dx,
$$
and
$$
\int_{\Omega_\ep} \p_x(T'_\Omega\le\tau_t\le t, x'_{\tau_t}\in \Omega)dx
$$
separately. These estimations are obtained in the following three lemmas, which in turn yields Theorem \ref{thm-heat-content-prime}. The proofs of these three lemmas are given in the following section.

\begin{lemma}\label{lemma-main-I}
Let $\Omega$, $\Omega_\ep$ and $x'_t$ be given as before. There exists a constant $C_1>0$ such that for $t>0$ small enough,
\[
\bigg|\int_{\Omega_\ep} \p_x(x'_{\tau_t}\in \Omega)dx-\Vol(\Omega_\ep)+\sqrt{\frac{2t}{\pi}}\sigma_0(\partial \Omega)-\frac{t}{4}\int_{\partial \Omega}H_{\pO,0}(s)d\sigma_0(s)\bigg|\le C_1t^{3/2}.
\]
\end{lemma}

\begin{lemma}\label{lemma-main-II}
 Let $\Omega$, $\Omega_\ep$ and $x'_t$ be as previously defined. Then
 \begin{align*}
\int_{\Omega_\ep} \p_x(T'_\Omega\le\tau_t\le t, x'_{\tau_t}\in \Omega)dx=o(t).
\end{align*}
\end{lemma}

\begin{lemma}\label{lemma-main-III}
Let $\Omega$, $\Omega_\ep$ and $x'_t$ be as previously defined. Then
\begin{align*}
\int_{\Omega_\ep} \p_x(\tau_t<T'_\Omega\le t)dx=o(t).
\end{align*}
\end{lemma}

\section{Proofs of the lemmas}\label{sec:proof}

In this final section, we prove Lemmas \ref{lemma-main-I}, \ref{lemma-main-II}, and \ref{lemma-main-III}. First let us recall notation. Let  $x'_t=\exp_x\left( -B^N_tN+B^T_tT+ A_tZ\right)$ be the truncated diffusion process, $X_t=(-B^N_t, B^T_t, A_t)$ the lift of $x'_t$ on $T_x\fH$, and $T'_\Omega=\inf\{t>0\,|\, {x}'_t\in \fH \setminus \Omega \}$ the exit time of $x'_t$ from $\Omega$.

To streamline the exposition, we defer the proof of several technical estimates in Subsection \ref{subsec:proof-of-first-lemma} to an appendix.

\subsection{Proof of Lemma \ref{lemma-main-I}}\label{subsec:proof-of-first-lemma}

From now on we denote $|(B^T_{\tau_t},A_{\tau_t})|=|B^T_{\tau_t}|^2+|A_{\tau_t}|$. Since $|(B^T_{\tau_t},A_{\tau_t})|\le ||X_{\tau_t}||^2$, from Lemma \ref{lemma-X-TD-prime} we know that for any $\delta>0$ ,
\[
\p_x(|(B^T_{\tau_t},A_{\tau_t})|>\delta)= O(e^{-\delta^2/4t}).
\]
Moreover, since
\begin{equation*}\begin{split}
\p_x(x'_{\tau_t}\in \Omega, |(B^T_{\tau_t},A_{\tau_t})|<\delta)
&\le \p_x(x'_{\tau_t}\in \Omega) \\
&\le\p_x(x'_{\tau_t}\in \Omega, |(B^T_{\tau_t},A_{\tau_t})|<\delta)+\p_x(|(B^T_{\tau_t},A_{\tau_t})|>\delta)
\end{split}\end{equation*}
we obtain that
\[
\p_x(x'_{\tau_t}\in \Omega)= \p_x(x'_{\tau_t}\in \Omega, |(B^T_{\tau_t},A_{\tau_t})|<\delta)+O(e^{-\delta^2/4t}).
\]
Letting $E(t)=\int_{\Omega_\ep}  \p_x(x'_{\tau_t}\in \Omega, |(B^T_{\tau_t},A_{\tau_t})|<\delta)dx$, we are reduced to prove
\begin{equation}\label{eq-E(t)}
E(t)=  \Vol(\Omega_\ep)-\sqrt{\frac{2t}{\pi}}\sigma_0(\partial \Omega)+\frac{t}{4}\int_{\partial \Omega}H_{\pO,0}(s)d\sigma_0(s)+ O(t^{3/2}).
\end{equation}
For fixed $x\in\Omega_\ep$, assume $d_{cc}(x,\partial\Omega)=r>0$. When $\ep>0$ is small enough, we can always assume that $x'_t$ started from $x\in\Omega_\ep$ stays inside the diffeomorphism neighborhood $\mathcal{O}_x$ of $\varphi_x$ within small time $t$. Hence we can consider the lifted process $X_t$. By comparing $\{x'_{\tau_t}\in \Omega, |(B^T_{\tau_t},A_{\tau_t})|<\delta\}$ and $\{B^N_{\tau_t}<r, |(B^T_{\tau_t},A_{\tau_t})|<\delta\}$ we have
\begin{equation*}\begin{split}
\p_x(x'_{\tau_t}\in \Omega, |(B^T_{\tau_t},A_{\tau_t})|<\delta)
&= \p_x(B^N_{\tau_t}<r, |(B^T_{\tau_t},A_{\tau_t})|<\delta) \\
& \quad - \p_x(B^N_{\tau_t}<r,x'_{\tau_t}\not\in \Omega,  |(B^T_{\tau_t},A_{\tau_t})|<\delta) \\
& \qquad + \p_x(B^N_{\tau_t}>r,x'_{\tau_t}\in \Omega,  |(B^T_{\tau_t},A_{\tau_t})|<\delta).
\end{split}\end{equation*}
We denote
\begin{align*}
&I_1(t)=\int_{\Omega_\ep} \p_x(B^N_{\tau_t}<r,|(B^T_{\tau_t},A_{\tau_t})|<\delta)dx, \\
&I_2(t)=\int_{\Omega_\ep} \p_x(B^N_{\tau_t}<r,x'_{\tau_t}\not\in \Omega, |(B^T_{\tau_t},A_{\tau_t})|<\delta)dx,\\
&I_3(t)=\int_{\Omega_\ep} \p_x(B^N_{\tau_t}>r,x'_{\tau_t}\in \Omega, |(B^T_{\tau_t},A_{\tau_t})|<\delta)dx.
\end{align*}
Then $E(t)=I_1(t)-I_2(t)+I_3(t)$. We estimate these terms in the following three steps. \\

\noindent {\bf Step 1:} First, let us estimate $I_1(t)$. Using the parametrization $\Psi$ from Lemma \ref{lemma-J}, we have
\[
I_1(t)=\int_0^\ep \int_{\partial \Omega} \p_x(B^N_{\tau_t}<r,|(B^T_{\tau_t},A_{\tau_t})|<\delta)J_\Psi(s,r)d\sigma_0(s)dr,
\]
where $x = \Psi(s,r)$. Furthermore, since
\[
 \p_x(B^N_{\tau_t}<r,|(B^T_{\tau_t},A_{\tau_t})|<\delta)=1- \p_x(|(B^T_{\tau_t},A_{\tau_t})|>\delta)- \p_x(B^N_{\tau_t}>r,|(B^T_{\tau_t},A_{\tau_t})|<\delta)
\]
and $\p_x(|(B^T_{\tau_t},A_{\tau_t})|>\delta)= O(e^{-\frac{\delta^2}{4t}})$,
we have
\begin{equation}\label{eq-I-1-J}
I_1(t)=\Vol(\Omega_\ep)-\int_0^\ep \int_{\partial \Omega} \p_x(B^N_{\tau_t}>r,|(B^T_{\tau_t},A_{\tau_t})|<\delta)J_\Psi(s,r)d\sigma_0(s)dr+o(t).
\end{equation}
Let $J(t)=\int_0^\ep\int_{\partial \Omega} \p_x(B^N_{\tau_t}>r,|(B^T_{\tau_t},A_{\tau_t})|<\delta)(1-r H_{\pO,0}(s))d\sigma_0(s)dr$. There exists $c>0$ depending on $\delta>0$ such that
\begin{equation*}\begin{split}
J(t)&=\int_0^\ep\int_{\partial \Omega} \p_x(B^N_{\tau_t}>r)(1-r H_{\pO,0}(s))d\sigma_0(s)dr+O(e^{-\frac{c}{t}})\\
&=  \int_0^\infty\int_{\partial \Omega} \p_x(B^N_{\tau_t}>r)(1-r H_{\pO,0}(s))d\sigma_0(s)dr-R_1(t)+ O(e^{-\frac{c}{t}}),
\end{split}\end{equation*}
where $R_1(t)=\int_\ep^\infty\int_{\partial \Omega} \p_x(B^N_{\tau_t}>r)(1-r H_{\pO,0}(s))d\sigma_0(s)dr$.
By Lemma \ref{lemma-joint} we can compute
\[
 \int_0^\infty\int_{\partial \Omega} \p_x(B^N_{\tau_t}>r)d\sigma_0(s)dr=\sigma_0(\partial \Omega) \int_0^\infty\frac{\sqrt{2}}{\sqrt{\pi t}}\int_r^{\infty} e^{-\frac{\xi^2}{2t}}d\xi dr=\frac{\sqrt{2t}}{\sqrt{\pi}}\sigma_0(\partial \Omega)
\]
and similarly
\[
 \int_0^\infty\int_{\partial \Omega} \p_x(B^N_{\tau_t}>r)r H_{\pO,0}(s)d\sigma_0(s)dr=\frac{t}{2}\int_{\partial \Omega} H_{\pO,0}(s)d\sigma_0(s).
\]
Therefore we have
\[
J(t)=\frac{\sqrt{2t}}{\sqrt{\pi}}\sigma_0(\partial \Omega)-\frac{t}{2}\int_{\partial \Omega} H_{\pO,0}(s)d\sigma_0(s)-R_1(t)+O(e^{-\frac{c}{t}}).
\]
Plugging this into \eqref{eq-I-1-J} yields
\[
I_1(t)=\Vol(\Omega_\ep)-\frac{\sqrt{2t}}{\sqrt{\pi}}\sigma_0(\partial \Omega)+\frac{t}{2}\int_{\partial \Omega} H_{\pO,0}(s)d\sigma_0(s)+R_1(t)+R_2(t)+O(e^{-\frac{c}{t}})
\]
where
\[
R_2(t)=\int_0^\ep \int_{\partial \Omega} \p_x(B^N_{\tau_t}>r,|(B^T_{\tau_t},A_{\tau_t})|<\delta)(1-J_\Psi(s,r)-r H_{\pO,0}(s))d\sigma_0(s)dr.
\]
Now we are left to estimate $R_1(t)$ and $R_2(t)$. Note when $r\ge \ep$ we have
\[
\bigg| 1-rH_{\pO,0}(s)\bigg| \le r^2|\ep^{-2}+K\ep^{-1}|,
\]
where $K=\max_{s\in\partial\Omega}|H_{\pO,0}(s)|$. Hence
\begin{equation*}\begin{split}
R_1(t) &= \frac{\sqrt{2}}{\sqrt{\pi t}}\int_\ep^\infty \int_r^{\infty} e^{-\frac{\xi^2}{2t}}d\xi\int_{\partial \Omega} (1-r H_{\pO,0}(s))d\sigma_0(s)  dr \\
&\le \frac{C_\ep}{\sqrt{t}} \sigma_0(\pO) \int_0^\infty \int_r^{\infty}r^2 e^{-\frac{\xi^2}{2t}}d\xi\, dr =O(t^{3/2}).
\end{split}\end{equation*}
For $R_2(t)$, by  \eqref{eq-jacobi} we know that
$
\bigg|1-J_\Psi(s,r)-rH_{\pO,0}(s) \bigg|\le K_1 r^2
$ for $r\in (0,\ep)$,
hence
\[
|R_2(t)|
\le \frac{\sqrt{2}K_1}{\sqrt{\pi t}} \sigma_0(\pO) \int_0^\infty\int_r^\infty r^2e^{-\frac{\xi^2}{2t}}d\xi dr=O(t^{3/2}).
\]
At the end we obtain
\begin{equation}\label{eq-I-1}
I_1(t)=\Vol(\Omega_\ep)-\frac{\sqrt{2t}}{\sqrt{\pi}}\sigma_0(\partial \Omega)+\frac{t}{2}\int_{\partial \Omega} H_{\pO,0}(s)d\sigma_0(s)+O(t^{3/2}).
\end{equation}

\noindent {\bf Step 2:} We are left to show that
$$
-I_2(t)+I_3(t)=-\frac{t}{4}\int_{\partial \Omega} H_{\pO,0}(s)d\sigma_0(s)+O(t^{3/2}).
$$
By changing coordinates we have
\begin{align*}
I_2(t)&=\int_0^\ep\int_{\partial \Omega} \p_x(B^N_{\tau_t}<r,x'_{\tau_t}\not\in \Omega, |(B^T_{\tau_t},A_{\tau_t})|<\delta)J_\Psi(s,r)d\sigma_0(s)dr.
\end{align*}
We claim that
\begin{equation}\label{eq-claim-A5}
\int_0^\ep\int_{\partial \Omega} \p_x(B^N_{\tau_t}<r,x'_{\tau_t}\not\in \Omega, |(B^T_{\tau_t},A_{\tau_t})|<\delta)(1-J_\Psi(s,r))d\sigma_0(s)dr=O(t^{3/2})
\end{equation}
and
\begin{equation}\label{eq-claim-A3}
\int_\ep^\infty\int_{\partial \Omega} \p_x(B^N_{\tau_t}<r,x'_{\tau_t}\not\in \Omega, |(B^T_{\tau_t},A_{\tau_t})|<\delta)d\sigma_0(s)dr=O(t^{3/2}).
\end{equation}
Estimates \eqref{eq-claim-A5} and \eqref{eq-claim-A3} are proved in sections \ref{App-claim-A5} and \ref{App-claim-A3} respectively. Then we have
\begin{align*}
I_2(t)&=\int_0^\infty \int_{\partial \Omega} \p_x(B^N_{\tau_t}<r,x'_{\tau_t}\not\in \Omega,  |(B^T_{\tau_t},A_{\tau_t})|<\delta)d\sigma_0(s)dr+O(t^{3/2}).
\end{align*}
Using the coordinate system in \eqref{eq-cartesian} and \eqref{eq-NTZ-cor} we have
\[
\{x'_{\tau_t}\not\in \Omega\}=\{h(B^T_{\tau_t}, A_{\tau_t};s)>r-B^N_{\tau_t}\},
\]
thus by Fubini we obtain
\begin{align}\label{eq-I2}
I_2(t)&=\int_0^\infty \int_{\partial \Omega} \p_x(B^N_{\tau_t}<r,h(B^T_{\tau_t}, A_{\tau_t};s)>r-B^N_{\tau_t},  |(B^T_{\tau_t},A_{\tau_t})|<\delta)d\sigma_0(s)dr+O(t^{3/2})\nonumber\\
&=\int_{\partial \Omega}  \,\E_x\left(h^+(B^T_{\tau_t},A_{\tau_t};s)\mathbbm{1}_{\{|(B^T_{\tau_t},A_{\tau_t})|<\delta\} }\right)\,d\sigma_0(s)
 +O(t^{3/2}),
\end{align}
where $h^+(y,z;s)=\int_0^\infty\mathbbm{1}_{\{h(y,z;s)>r\}}dr$ is the positive part of $h(y,z;s)$.\\

\noindent {\bf Step 3:} Now consider $I_3(t)$. Note
\begin{align*}
I_3(t)&=\int_0^\ep\int_{\partial \Omega} \p_x(B^N_{\tau_t}>r,x'_{\tau_t}\in  \Omega, |(B^T_{\tau_t},A_{\tau_t})|<\delta)d\sigma_0(s) dr\\
&-\int_0^\ep\int_{\partial  \Omega} \p_x(B^N_{\tau_t}>r,x'_{\tau_t}\in  \Omega, |(B^T_{\tau_t},A_{\tau_t})|<\delta)(1-J_\Psi(s,r))d\sigma_0(s)dr.
\end{align*}
We claim that
\begin{equation}\label{eq-minor-I3-1}
\int_0^\ep\int_{\partial  \Omega} \p_x(B^N_{\tau_t}>r,x'_{\tau_t}\in  \Omega, |(B^T_{\tau_t},A_{\tau_t})|<\delta)(1-J_\Psi(s,r))d\sigma_0(s)dr=O(t^{3/2}),
\end{equation}
and
\begin{equation}\label{eq-minor-I3-2}
\int_\ep^\infty\int_{\partial \Omega} \p_x(B^N_{\tau_t}>r,x'_{\tau_t}\in \Omega, |(B^T_{\tau_t},A_{\tau_t})|<\delta)d\sigma_0(s)dr=O(t^{3/2}).
\end{equation}
Estimates \eqref{eq-minor-I3-1} and \eqref{eq-minor-I3-2} are proved in sections \ref{App-minor-I3-1} and \ref{App-minor-I3-2} respectively.
Therefore by  \eqref{eq-cartesian} and \eqref{eq-NTZ-cor} we have
\begin{align*}
I_3(t)&=\int_0^\infty\int_{\partial \Omega} \p_x(B^N_{\tau_t}>r,x'_{\tau_t}\in \Omega, |(B^T_{\tau_t},A_{\tau_t})|<\delta)d\sigma_0(s)dr+O(t^{3/2})\\
&=\int_0^\infty \int_{\partial \Omega} \p_x(B^N_{\tau_t}>r,h(B^T_{\tau_t}, A_{\tau_t};s)<r-B^N_{\tau_t}, |(B^T_{\tau_t},A_{\tau_t})|<\delta)d\sigma_0(s)dr+O(t^{3/2})
\end{align*}
Let $h^-(y,z;s)=|h(y,z;s)| - h^+(y,z;s)$ be the negative part of $h(y,z;s)$, then by Fubini,
\[
I_3(t)=\int_{\partial \Omega}  \,\E_x\left(\min\left(h^-(B^T_{\tau_t},A_{\tau_t};s),B^N_{\tau_t}\right)\mathbbm{1}_{\{|(B^T_{\tau_t},A_{\tau_t})|<\delta\} }\right)\,d\sigma_0(s)
 +O(t^{3/2})
\]
Note $\min\left(h^-(B^T_{\tau_t},A_{\tau_t};s),B^N_{\tau_t}\right)=h^-(B^T_{\tau_t},A_{\tau_t};s)+\min\left(B^N_{\tau_t}-h^-(B^T_{\tau_t},A_{\tau_t};s),0\right)$, hence
\[
I_3(t)=\int_{\partial \Omega}  \,\E_x\left(h^-(B^T_{\tau_t},A_{\tau_t};s)\mathbbm{1}_{\{|(B^T_{\tau_t},A_{\tau_t})|<\delta\} }\right)\,d\sigma_0(s)+R_3(t)
 +O(t^{3/2}),
\]
where
\begin{align*}
R_3(t)&=\int_{\partial \Omega}  \,\E_x\left(\min\left(B^N_{\tau_t}-h^-(B^T_{\tau_t},A_{\tau_t};s),0\right)\mathbbm{1}_{\{|(B^T_{\tau_t},A_{\tau_t})|<\delta\} }\right)\,d\sigma_0(s).
\end{align*}
We claim that
\begin{align}\label{eq-C-4}
|R_3(t)|=O(t^{3/2});
\end{align}
see section \ref{App-C-4} for a proof. At the end  we have
\begin{align}\label{eq-I3}
I_3(t)&=\int_{\partial \Omega}  \,\E_x\left(h^-(B^T_{\tau_t},A_{\tau_t};s)\mathbbm{1}_{\{|(B^T_{\tau_t},A_{\tau_t})|<\delta\} }\right)\,d\sigma_0(s) +O(t^{3/2}).
\end{align}

Now by combining \eqref{eq-I2} and \eqref{eq-I3} we obtain
\begin{align*}
-I_2(t)+I_3(t)&=-\int_{\partial \Omega}  \,\E_x\left(h(B^T_{\tau_t},A_{\tau_t};s)\mathbbm{1}_{\{|(B^T_{\tau_t},A_{\tau_t})|<\delta\} }\right)\,d\sigma_0(s)
+O(t^{3/2})\\
&=-\int_{\partial \Omega}  \,\E_x\left(\frac12 {H_{\pO,0}(s)}(B^T_{\tau_t})^2+k_1(s)A_{\tau_t}\right)\,d\sigma_0(s)+C_1(t)+C_2(t)+O(t^{3/2})
\end{align*}
where
\[
C_1(t)=\int_{\partial \Omega}  \,\E_x\left(\left(\frac12{H_{\pO,0}(s)}(B^T_{\tau_t})^2+k_1(s)A_{\tau_t}-h(B^T_{\tau_t},A_{\tau_t};s)\right)\mathbbm{1}_{\{|(B^T_{\tau_t},A_{\tau_t})|<\delta\} }\right)\,d\sigma_0(s)
\]
and
\[
C_2(t)=\int_{\partial \Omega}  \,\E_x\left(\left(\frac12{H_{\pO,0}(s)}(B^T_{\tau_t})^2+k_1(s)A_{\tau_t}\right)\mathbbm{1}_{\{|(B^T_{\tau_t},A_{\tau_t})|\ge\delta\} }\right)\,d\sigma_0(s)\]
We claim that for sufficiently small $\eta>0$,
\begin{align}\label{eq-C}
C_1(t)=O(t^{3/2-\eta}), \quad C_2(t)=O(t^{3/2-\eta}).
\end{align}
Estimate \ref{eq-C} will be proved in section \ref{App-C}. Then we have
\begin{align*}
-I_2(t)+I_3(t)&=-\bigg(\int_{\partial \Omega} \frac12{H_{\pO,0}(s)} d\sigma_0(s)\bigg)\cdot\E\left((B^T_{\tau_t})^2\right)-\bigg(\int_{\partial \Omega}k_1(s) d\sigma_0(s)\bigg) \cdot\E\left(A_{\tau_t}\right)+o(t).
\end{align*}
Moreover, since
\[
\E\left((B^T_{\tau_t})^2\right)=\int_0^t\E\left((B^T_{\tau_t})^2|\tau_t=\tau\right)\p(\tau_t=d\tau)=\int_0^t \tau\, \p(\tau_t=d\tau)=\frac{t}{2}
\]
and
\[
\E\left(A_{\tau_t}\right)=\E\bigg(-B^N_{\tau_t}B^T_{\tau_t}+2\int_0^{\tau_t}B_s^NdB_s^T\bigg)=0,
\]
we obtain
\begin{equation}\label{I2I3-final}
-I_2(t)+I_3(t)=-\frac{t}{4}\int_{\partial \Omega}H_{\pO,0}(s) d\sigma_0(s)+O(t^{3/2}).
\end{equation}
Combining \eqref{I2I3-final} with \eqref{eq-I-1} we have
\[
E(t)=I_1(t)-I_2(t)+I_3(t)=\Vol(\Omega_\ep)-\sqrt{\frac{2t}{\pi}}\sigma_0(\partial \Omega)+\frac{t}{4}\int_{\partial \Omega}H_{\pO,0}(s)d\sigma_0(s)+ O(t^{3/2}),
\]
which completes the proof of \eqref{eq-E(t)}.

\subsection{Proof of Lemma \ref{lemma-main-II}}

\begin{lemma}\label{lemma-sigma-w}
Let $\Omega$ and $\Omega_\ep$ be as previously defined, and recall that $\pO$ is locally parameterized by a function $h$ as in \eqref{eq-NTZ-cor}. There is a constant $K>0$ so that
for any $\delta<\ep$
and any $x\in \Omega_\delta$, $\sigma\in\partial \Omega$, and $w\in\fH$ such that
\begin{itemize}
\item $d_{cc}(\sigma, x)<\delta$ and $d_{cc}(w,x)<\delta$,
\item the estimate
\begin{equation}\label{eq-w-lambda}
\lambda_1-w_1\le h_y(\lambda_2,\lambda_3;s)(w_2-\lambda_2)
-K\left(|w_2-\lambda_2|^2+|w_3-\lambda_3|\right),
\end{equation}
holds, where $\varphi_x^{-1} (\sigma)=(-\lambda_1,\lambda_2,\lambda_3)$ and $\varphi_x^{-1} (w)=(-w_1, w_2, w_3)$,
\end{itemize}
then the following conclusions hold:
\begin{itemize}
\item[(1)] $w\not\in \Omega$,
\item[(2)] $|h_y(\lambda_2, \lambda_3;s)|\le K (|\lambda_2|+|\lambda_3|)$.
\end{itemize}
\end{lemma}

\begin{proof}
Since $\sigma \in \pO$ is sufficiently close to $x$, we know that
\begin{equation}\label{lambdas}
-\lambda_1=h(\lambda_2,\lambda_3;s)-r.
\end{equation}
Using the Taylor expansion of $h$ we have
\begin{align*}
h(w_2, w_3;s)&=h(\lambda_2,\lambda_3;s)+h_y(\lambda_2,\lambda_3;s)(w_2-\lambda_2) 
+R(s,w', \lambda'),
\end{align*}
where $\lambda'=(\lambda_2,\lambda_3)$ and $w'=(w_2,w_3)$. Since $\pO$ is $C^3$ the remainder term can be bounded uniformly:
\[
|R(s,w', \lambda')|\le \frac{1}{2}K(|w_2-\lambda_2|^2+|w_3-\lambda_3|)
\]
for a suitable choice of $K$.

We first verify (1). It follows from the preceding estimates that
\[
h(w_2, w_3;s)\ge h(\lambda_2,\lambda_3;s)+h_y(\lambda_2,\lambda_3;s)(w_2-\lambda_2)
-\frac12K(|w_2-\lambda_2|^2+|w_3-\lambda_3|).
\]
Together with \eqref{eq-w-lambda} and \eqref{lambdas} this then implies that $h(w_2,w_3;s)\ge -w_1+r$, namely $w\not\in\Omega$.

(2) is an easy consequence of the fact that $h$ is $C^2$ and $h_y(0,0;s)=0$.
\end{proof}

\begin{proof}[Proof of Lemma \ref{lemma-main-II}]
We denote by $\sigma\in\partial \Omega$ the exit point $x'_{T'_{\Omega}}$, and let $\p_\sigma$ be the probability measure of the Markov process $x'_t$ started from $\sigma$. Recalling the notation $X_t=\varphi_x^{-1}(x'_t)$, we estimate
\begin{align*}
\p_x(T'_\Omega\le\tau_t\le t, x'_{\tau_t}\in \Omega)
&\le \p_x(|X_{T'_\Omega}|>\delta)+\p_x(T'_\Omega\le\tau_t\le t, |X_{T'_\Omega}|<\delta, x'_{\tau_t}\in \Omega)\\
&\le \p_x(T'_\Omega\le\tau_t\le t, |X_{T'_\Omega}|<\delta, x'_{\tau_t}\in \Omega)+O(t^{3/2}).
\end{align*}
For fixed $\sigma\in\partial \Omega$, $u\ge0$, let $\phi_{x}( \sigma,u)$ be the probability that the process $x'_t$ started from $\sigma$ has farthest achievement along the horizontal normal direction $N$ up to time $u$ inside $\Omega$, that is,
\[
\phi_{x}( \sigma,u)=\p_\sigma(x'_{\tau_u}\in \Omega), \quad \tau_u=\inf{_\sigma}\lbrace\tau: B^N_\tau=\sup_{0\le v\le u}B^N_v\rbrace.
\]
Under $\p_x$, we have $\tau_u=\inf\lbrace\tau: B^N_\tau=\sup_{T'_\Omega\le v\le u}B^N_v\rbrace.$
Note $T'_\Omega\le\tau_t\le t$ means that  $x'_t$ hits $\partial \Omega$ before $B^N_t$ achieve its maximum, hence
\begin{align}\label{eq-T-X-D}
\p_x(T'_\Omega\le\tau_t\le t, |X_{T'_\Omega}|<\delta, x'_{\tau_t}\in \Omega)
\le \E_x\left({\mathbbm{1}}_{\{T'_\Omega\le t, |X_{T'_\Omega}|<\delta\}}\phi_{x}( x'_{T'_\Omega}, t-T'_\Omega) \right).
\end{align}
Using the same notation as in Lemma \ref{lemma-sigma-w}, for $\sigma\in\partial \Omega$, $d_{cc}(\sigma,x)<\delta$, we have $-\lambda_1=h(\lambda_2,\lambda_3;s)-r$, that is
\[
\varphi_x^{-1}(\sigma)=(h(\lambda_2,\lambda_3;s)-r, \lambda_2,\lambda_3).
\]
Also, under $\p_\sigma$ we can write
\begin{align*}
X_{\tau_u}=\varphi_x^{-1}(x'_{\tau_u})
=(-\lambda_1-\beta^N_{\tau_u}, \lambda_2+\beta^T_{\tau_u}, \lambda_3+\mathcal{A}_{\tau_u}),
\end{align*}
where $\beta^N$ and $\beta^T$ are independent standard Brownian motions under $\p_\sigma$, and $\mathcal{A}_{\tau_u}=-\beta^N_{\tau_u}\lambda_2+\beta^T_{\tau_u}\lambda_1-\beta^T_{\tau_u}\beta^N_{\tau_u}+2\int_0^{\tau_u}\beta^N_sd\beta^T_s$.
By Lemma \ref{lemma-sigma-w} we know that if $x'_{\tau_u}\in \Omega$ then
\begin{align*}
-\beta^N_{\tau_u}&\ge h_y(\lambda_2,\lambda_3;s)\beta^T_{\tau_u}-K\left(|\beta^T_{\tau_u}|^2+|\mathcal{A}_{\tau_u}|\right).
\end{align*}
Hence for $\sigma=x'_{T'_\Omega}$, $u=t-T'_\Omega$,
$$
\phi_x(\sigma, u)=\p_\sigma(x'_{\tau_u}\in \Omega) \le\p_\sigma\bigg(\beta^N_{\tau_u}\le -h_y(\lambda_2,\lambda_3;s) \beta^T_{\tau_u}+K\left(|\beta^T_{\tau_u}|^2+|\mathcal{A}_{\tau_u}|\right) \bigg).
$$
Since $\int_0^{\tau_u}\beta^N_sd\beta^T_s$ is a martingale, it can be written as a time changed Brownian motion $\beta_{\int_0^{\tau_u}(\beta_s^N)^2ds}$ where $\beta$ is an independent standard Brownian motion. Hence there exists $C>0$ such that
\begin{equation*}\begin{split}
\p_\sigma\bigg(\bigg|\int_0^{\tau_u}\beta^N_sd\beta^T_s \bigg|>\beta^N_{\tau_u}{\tau_u}^{1/2-\eta}\bigg)
&=\p_\sigma\bigg(|\beta_1|^2\int_0^{\tau_u}(\beta^N_s)^2ds >(\beta^N_{\tau_u})^2{\tau_u}^{1-2\eta}\bigg) \\
&\le \p_\sigma\bigg(|\beta_1|^2\cdot{\tau_u} >{\tau_u}^{1-2\eta}\bigg) \\
&= \p_\sigma\bigg(|\beta_1|>\frac{1}{t^{\eta}}\bigg) =O(e^{-C/t^{2\eta}})
\end{split}\end{equation*}
for some $\eta\in(0,1/2)$. Therefore we know that
\begin{align*}
&\phi_x(\sigma, u)\le
 \p_\sigma\bigg(\beta^N_{\tau_u}\big(1-K|\lambda_2|-K|\beta^T_{\tau_u}|-2K\tau_u^{1/2-\eta}\big)\le |h_y(\lambda_2,\lambda_3;s)||\beta^T_{\tau_u}|
+K|\lambda_1||\beta^T_{\tau_u}|\\
& \qquad \quad +K |\beta^T_{\tau_u}|^2
 \bigg)+O(t^{3/2}),
\end{align*}
Hence we have
\begin{align*}
&\phi_x(\sigma, u)\le
 \p_\sigma\bigg(\frac{1}{2}\beta^N_{\tau_u}\le |h_y(\lambda_2,\lambda_3;s)||\beta^T_{\tau_u}|
+K|\lambda_1||\beta^T_{\tau_u}|+K |\beta^T_{\tau_u}|^2
 \bigg)\\
&\qquad \quad  +\p_\sigma\left(1-K|\lambda_2|-K|\beta^T_{\tau_u}|-2K\tau_u^{1/2-\eta}<\frac{1}{2}\right)+O(t^{3/2}).
\end{align*}
Recall $d_{cc}(\sigma,x)<\delta$ implies $|\lambda_2|<\delta$. When $\delta>0$ is small enough such that $\delta\le\frac{1}{8K}$, and when $t$ is small enough such that $2Kt^{1/2-\eta}<1/4$, in the set $\{||X_{T'_\Omega}||<\delta\}\cap\{T'_\Omega<t \}\cap\{\tau_u<t\}$ we have
\begin{equation*}\begin{split}
\p_\sigma\left(1-K|\lambda_2|-K|\beta^T_{\tau_u}|-2K\tau_u^{1/2-\eta}<\frac{1}{2}\right)
&\le\p_\sigma\left(K|\lambda_2|+K|\beta^T_{\tau_u}|>\frac{1}{4}\right) \\
&\le \p_\sigma\left(|\beta^T_{\tau_u}|_*>\frac{1/4-K\delta}{K}\right) \\
&\le  \p_\sigma\left(|\beta^T_t|_*>\frac{1}{8K}\right)=O(e^{-C_K/t})
\end{split}\end{equation*}
for some $C_K>0$, where $|\beta^T|_*$ is the running maximum of $|\beta^T|$. Hence we have in the set $\{||X_{T'_\Omega}||<\delta\}\cap\{T'_\Omega<t \}\cap\{\tau_u<t\}$,
\begin{equation*}\begin{split}
\phi_x(\sigma, u)=\p_\sigma(x'_{\tau_u}\in \Omega)
&\le \p_\sigma\bigg(\beta^N_{\tau_u}/2\le |h_y(\lambda_2,\lambda_3;s)||\beta^T_{\tau_u}|
+K|\lambda_1||\beta^T_{\tau_u}|+K |\beta^T_{\tau_u}|^2  \bigg)+O(t^{3/2})\\
&= \int_0^ud\tau\int_{-\infty}^\infty  \frac{1}{\sqrt{2\pi\tau}}e^{-\frac{y^2}{2\tau}} dy \int_{\{0<\xi< 2F(y,\tau) \}}  \Phi(\xi, \tau;u)d\xi+O(t^{3/2})\\
&=  \int_0^ud\tau\int_{-\infty}^\infty  \frac{1}{\pi\sqrt{\tau}\sqrt{u-\tau}}\left(1- e^{\frac{1}{\tau}F(y,\tau)}\right)  \frac{1}{\sqrt{2\pi\tau}}e^{-\frac{y^2}{2\tau}} dy+O(t^{3/2})
\end{split}\end{equation*}
where
\[
F(y,\tau)=\bigg(|h_y(\lambda_2,\lambda_3;s)||y|+K|\lambda_1||y|+K |y|^2\bigg)^2.
\]
Since
\[
1-e^{\frac{1}{\tau}F(y,\tau)}
\le
\frac{C_{K'}}{\tau}\left(y^4+|h_y|^2|y|^2+|y|^2\lambda_1^2\right)
\]
for some $C_{K'}>0$, we have for $\sigma=x'_{T'_\Omega}$, $u=t-T'_\Omega$, that
\begin{align*}
\phi_x(\sigma, u)
&\le\int_0^ud\tau\int_{-\infty}^\infty \frac{C_{K'}}{\pi{\tau}^{3/2}\sqrt{u-\tau}}\left(y^4+|h_y|^2|y|^2+|y|^2\lambda_1^2\right)  \frac{1}{\sqrt{2\pi\tau}}e^{-\frac{y^2}{2\tau}} dy+O(t^{3/2})\\
&=C_{K'}\int_0^u \frac{3\sqrt{\tau}+| h_y|^2\tau^{-1/2}+\tau^{-1/2}\lambda_1^2}{\sqrt{u-\tau}} d\tau+O(t^{3/2})\\
&\le C'\bigg(u +|h_y|^2+\lambda_1^2\bigg)+O(t^{3/2}) \\
&\le C'' \left(t+K^2(|\lambda_2|^2+|\lambda_3|)+\lambda_1^2\right)+O(t^{3/2})
\end{align*}
The last inequality is due to Lemma \ref{lemma-sigma-w}. Plugging back into  \eqref{eq-T-X-D} we obtain
\begin{align*}
&\p_x(T'_\Omega\le\tau_t\le t, x'_{\tau_t}\in \Omega) \\
&\quad \le  \E_x\left({\mathbbm{1}}_{\{T'_\Omega\le t, ||X_{T'_\Omega}||<\delta\}}C''  \left(t+K^2(|B^T_{T'_\Omega}|^2+|A_{T'_\Omega}|)+|B^N_{T'_\Omega}|^2\right)  \right)+O(t^{3/2})\\
&\quad \le  \E_x\left({\mathbbm{1}}_{\{T'_\Omega\le t\}} \left(C'' t+C''K^2\left(\sup_{0\le s\le t}|B^T_{s}|^2+\sup_{0\le s\le t}|A_{s}|\right)+C''\sup_{0\le s\le t}|B^N_{s}|^2\right)  \right)+O(t^{3/2})\\
&\quad \le C'' \p_x(T'_\Omega\le t)t+O(t^{3/2}) \\
& \qquad \quad + C'''\E_x\left({\mathbbm{1}}_{\{T'_\Omega\le t\}}\right)^{1/2}\left[\left(\E_x\left(\sup_{0\le s\le t}|B^T_{s}|^4\right)\right)^{1/2}+\left(\E_x\left(\sup_{0\le s\le t}|A_{s}|^2\right) \right)^{1/2}+\left(\E_x\left(\sup_{0\le s\le t}|B^N_{s}|^4\right)\right)^{1/2}\right]
\end{align*}
By Doob's maximal inequality, for $i=N, T$,
\[
\E_x\left(\sup_{0\le s\le t}|B^i_{s}|^4\right)\le (4/3)^4\E_x(|B_t^i|^4)=4^43^{-3}t^2
\]
and
\begin{align*}
\E_x\left(\sup_{0\le s\le t}|A_{s}|^2\right)\le 4\E_x(|A_t|^2)=4t^2.
\end{align*}
Therefore we obtain
\begin{align*}
&\p_x(T'_\Omega\le\tau_t\le t, x'_{\tau_t}\in \Omega)
\le C''\p_x(T'_\Omega\le t) t+C'''\p_x\left(T'_\Omega\le t\right)^{1/2}t+O(t^{3/2})
\end{align*}
From Lemma \ref{lemma-X-TD-prime} we know there exists $C^*,C_2, c>0$ such that
\begin{align*}
& \int_{\Omega_\ep} \p_x(T'_\Omega\le\tau_t\le t, x'_{\tau_t}\in \Omega)dx \\
& \quad \le C^*\int_{\Omega_\ep}\left(te^{-\frac{d_{cc}(x,\pO)^2}{ct}} +te^{-\frac{d_{cc}(x,\pO)^2}{2ct}}\right)dx+O(t^{3/2})\\
& \quad =C^* \int_{0}^\ep \int_{\partial \Omega}\left(te^{-\frac{r^2}{ct}} +te^{-\frac{r^2}{2ct}}\right)(1-H_{\pO,0}(s)r)d\sigma_0(s)dr+O(t^{3/2}) \\
& \quad \le C_2t^{3/2}.
\end{align*}
Hence we obtain Lemma \ref{lemma-main-II}.
\end{proof}

\subsection{Proof of Lemma \ref{lemma-main-III}}
Our main task is to estimate $\p_x(\tau_t<T'_\Omega\le t)$, namely the probability of $x'_t$
\begin{itemize}
\item
remaining inside $\Omega$ up to its furthest excursion along the outward horizontal normal direction $-N$ with in time $t$, and
\item
exiting $\Omega$ after the ``maximum excursion" along $-N$ before $t$.
\end{itemize}
Again, we deal with the lifted process on the tangent space.
Let $X_t=(-B^N_t, B^T_t, A_t)$ be the Markov process as given in \eqref{eq-X-NTZ}. For any $w\in\mathcal{O}_x$, we denote $\varphi_x^{-1}(w)=(-w_1, w_2, w_3)$. From \eqref{eq-NTZ-cor} we know  that $w\in  \Omega$ if $w_1<r-h(w_2,w_3;s)$ and  $w\not\in \Omega$ if $w_1>r-h(w_2,w_3;s)$. By Lemma \ref{lemma-h-H} we can then conclude that there exists a $\delta\in(0, \ep)$,
$$
w_1<r-\frac{1}{2}H_{\pO,0}(s)w_2^2-k_1(s)w_3+\frac{1}{\delta^2}|(w_2,w_3)|^3
$$
if $w\in  \Omega$, while
$$
w_1\ge r-\frac{1}{2}H_{\pO,0}(s)w_2^2-k_1(s)w_3-\frac{1}{\delta^2}|(w_2,w_3)|^3
$$
if $w\not\in  \Omega$. Hence in probabilistic language we have
\[
\p_x(\tau_t<T'_\Omega\le t, \sup_{0\le s\le t}||X_s||\le \delta)\le D(s,r,t),
\]
where
\begin{align*}
 D(s,r,t)&=\p_x\bigg( B^N_{\tau_t}<r-\frac{1}{2}H_{\pO,0}(s) |B^T_{\tau_t}|^2-k_1(s) A_{\tau_t}+\frac{1}{\delta^2}|( B^T_{\tau_t}, A_{\tau_t})|^3,\\
& \exists v\in [\tau_t,t]:  B^N_{v}\ge r-\frac{1}{2}H_{\pO,0}(s) |B^T_{v}|^2-k_1(s) A_{v}-\frac{1}{\delta^2}|( B^T_{v}, A_{v})|^3\bigg) \, .
\end{align*}
\begin{proof}[Proof of Lemma \ref{lemma-main-III}]
By changing coordinates, we have
\begin{equation}\label{eq-int-p}
\int_{\Omega_\ep} \p_x(\tau_t<T'_\Omega\le t)dx\le(1+K'\ep)\int_0^\ep\int_{\partial \Omega}\p_{(s,r)}(\tau_t<T'_\Omega\le t)d\sigma_0(s)dr,
\end{equation}
where $K'=\max_{s\in \partial \Omega}|H_{\pO,0}(s)|+K_1$ and $K_1$ is as in Lemma \ref{lemma-J}. For fixed $s\in \partial \Omega$ we want to bound
$\int_0^\ep\p_{(s,r)}(\tau_t<T'_\Omega\le t)dr$.
From Lemma \ref{lemma-X-TD-prime} we know that there exists $C>0$ such that
\begin{align}\label{eq-P-D}
\p_x(\tau_t<T'_\Omega\le t)\le  \p_x\left(\sup_{0\le s\le t}||X_s||\ge \delta\right)+D(s,r,t)\le D(s,r,t)+Ct^{3/2}.
\end{align}
The rest of proof is then devoted into the estimate of $D(s,r, t)$.

For each $v\in [\tau_t,t]$, introduce a parameter $\tau = \tau(v) \in[0,1]$ such that $v=\tau_t+\tau(t-\tau_t)$, and let $(M^N_\tau, M^T_\tau,M^{A}_\tau)$ be given as
\begin{align*}
& M^N_\tau=\frac{B^N_{\tau_t}-B^N_{\tau_t+\tau(t-\tau_t)}}{\sqrt{t-\tau_t}},\quad M^T_\tau=\frac{B^T_{\tau_t+\tau(t-\tau_t)}-B^T_{\tau_t}}{\sqrt{t-\tau_t}}.
\end{align*}
Clearly $M^N_\tau$ is a Brownian meander process. Due to independence of $B^T$ and $(B^N,\tau_t)$ we know that $M^T_\tau$ is an independent standard Brownian motion process. We have
\[
X_v=X_{\tau_t}+\bigg(\sqrt{t-\tau_t}M^N_\tau,  \sqrt{t-\tau_t}M^T_\tau , A_{\tau_t+\tau(t-\tau_t)}-A_{\tau_t} \bigg),
\]
where
$$
A_{\tau_t+\tau(t-\tau_t)}-A_{\tau_t}= -\sqrt{t-\tau_t}\left(M_\tau^NB_{\tau_t}^T-M_\tau^TB_{\tau_t}^N \right)-(t-\tau_t)M^N_\tau M^T_\tau+2(t-\tau_t)\int_0^\tau M_s^NdM_s^T.
$$

Let $\chi_t(\xi, y, z, u)$ be the density function of $(B^N_{\tau_t}, B^T_{\tau_t}, A_{\tau_t},t-\tau_t)$, then
\begin{align*}
&D(s,r,t)=\int_0^\infty d\xi\int_{-\infty}^{\infty}dy \int_{-\infty}^\infty dz \int_0^t du\, \chi_t(\xi, y, z, u){\mathbbm{1}}_{\xi-r+\frac{1}{2}H_{\pO,0}(s)|y|^2+k_1(s)z<\frac{1}{\delta^2}|(y,z)|^3} \\
&\quad \cdot\p_x\bigg( \exists \tau\in [0,1]: \xi-\sqrt{u}M^N_{\tau}\ge r-\frac{1}{2}H_{\pO,0}(s)|y+\sqrt{u}M^T_{\tau}|^2 -k_1(s) \bigg(z-\sqrt{u}(M^N_\tau y-M^T_\tau \xi)-uM^N_\tau M^T_\tau\\
&\qquad +2u\int_0^\tau M_s^NdM_s^T\bigg)
-\frac{1}{\delta^2}\bigg|\bigg(y+\sqrt{u}M^T_{\tau},z-\sqrt{u}(M^N_\tau y-M^T_\tau \xi)-{u}M^N_\tau M^T_\tau+2u\int_0^\tau M_s^NdM_s^T\bigg)\bigg|^3\bigg).
\end{align*}
Here recall the notation $|(y,z)|=\sqrt{y^2+|z|}$. Furthermore since $|(y_1+y_2,z_1+z_2)|^3\le 8|(y_1,z_1)|^3+8|(y_2,z_2)|^3$ for any $y_1, y_2, z_1, z_2 \in\R$, we have
\begin{align}\label{eq-D-W}
 D(s,r,t)&\le\int_0^\infty d\xi\int_{-\infty}^{\infty}dy  \int_{-\infty}^\infty dz \int_0^t du\, \chi_t(\xi, y, z, u){\mathbbm{1}}_{\xi-r+\frac{1}{2}H_{\pO,0}(s)|y|^2+k_1(s) z<\frac{1}{\delta^2}|(y,z)|^3}\nonumber \\
& \quad \cdot\p_x\bigg( \exists \tau\in [0,1]: \xi-\sqrt{u}M^N_{\tau}\ge r-\frac{1}{2}H_{\pO,0}(s)|y|^2-k_1(s)z-\frac{8}{\delta^2} |(y,z)|^3-\sqrt{u}W(y,u,z,\tau)\bigg)
\end{align}
where
\begin{align*}
&W(y,u,z,\tau)=H_{\pO,0}(s)|M^T_{\tau}||y|+\frac{1}{2}H_{\pO,0}(s)\sqrt{u}|M^T_{\tau}|^2 \\
& \quad +k_1(s)\left(-(M^N_\tau y-M^T_\tau \xi)-\sqrt{u}M^N_\tau M^T_\tau+2\sqrt{u}\int_0^\tau M_s^NdM_s^T\right)\\
& \qquad +\frac{8}{\delta^2}{u}\bigg|\bigg(M^T_{\tau},-(M^N_\tau y-M^T_\tau \xi)-\sqrt{u}M^N_\tau M^T_\tau+2\sqrt{u}\int_0^\tau M_s^NdM_s^T\bigg)\bigg|^3.
\end{align*}
To estimate $W(y,u,z, \tau)$. First we prove the following lemma.

\begin{lemma}\label{lemma-MN-MT}
Let $M^N_\tau$ and $M^T_\tau$, $0\le \tau\le1$ be as before. Then for any $0<\eta<1/2$, there exists $\eta'>0$ such that for any $u\in[0,t)$,
\begin{equation}\label{eq-int-MT}
\p\bigg(\exists \tau\in[0,1],\,{u}^{2\eta}\bigg|\int_0^\tau M_s^NdM_s^T\bigg| >|M^T_\tau|_*\bigg)=O( e^{-c/t^{\eta'}})
\end{equation}
where $|M^T_\tau|_*=\sup_{0\le s\le \tau}|M^T_s|$ is the running maximum. Moreover we have
\[
\p\bigg(\exists \tau\in[0,1],\,{u}^{2\eta}\bigg|M^N_\tau M^T_\tau-2\int_0^\tau M_s^NdM_s^T\bigg| >|M^T_\tau|_*\bigg)=O( e^{-c/t^{\eta'}}).
\]
\end{lemma}

\begin{proof}
First note
\begin{align*}
&\p\bigg(\exists \tau\in[0,1],\,{u}^{2\eta}\bigg|M^N_\tau M^T_\tau-2\int_0^\tau M_s^NdM_s^T\bigg| >|M^T_\tau|_*\bigg)\\
& \quad \le \p\bigg(\exists \tau\in[0,1],\,{u}^{2\eta}|M^N_\tau M^T_\tau|_* >\frac12|M^T_\tau|_*\bigg) \\
& \qquad +\p\bigg(\exists \tau\in[0,1],\,2{u}^{2\eta}\bigg|\int_0^\tau M_s^NdM_s^T\bigg| >\frac12|M^T_\tau|_*\bigg)
\end{align*}
and
\[
\p\bigg(\exists \tau\in[0,1],\,{u}^{2\eta}|M^N_\tau |_* >\frac12\bigg)\le\p\bigg({u}^{2\eta}|M^N_1 |_* >\frac12\bigg)=O(e^{-1/2u^{2\eta}}).
\]
We just need to prove \eqref{eq-int-MT}. Since $\int_0^\tau M_s^NdM_s^T$ is a martingale, it can be written as a time changed Brownian motion, namely
\[
\int_0^\tau M_s^NdM_s^T=\beta_{\int_0^\tau (M^N_s)^2ds},
\]
where $\beta$ is an independent standard Brownian motion. Therefore
\begin{align*}
&\p\bigg(\exists \tau\in[0,1],\,{u}^{2\eta}\bigg|\int_0^\tau M_s^NdM_s^T\bigg| >|M^T_\tau|_*\bigg) \\
& \quad =\p\bigg(\exists \tau\in[0,1],\,{u}^{2\eta}\left(\int_0^\tau (M^N_s)^2ds\right)^{1/2} |\beta_1| >\sqrt{\tau}|M^T_1|_*\bigg)\\
& \qquad \le\p\bigg(\exists \tau\in[0,1],\,{u}^{\eta}\left(\int_0^\tau (M^N_s)^2ds\right)^{1/2} >\sqrt{\tau}\bigg)+\p\bigg(u^\eta|\beta_1|>|M^T_1|_*  \bigg).
\end{align*}
We easily observe that there exists $c, c', \eta'>0$ such that
\[
\p\bigg(u^\eta|\beta_1|>|M^T_1|_*  \bigg)\le \p(|\beta_1|>u^{-\eta/2})+\p(|M_1^T|_*<u^{\eta/2})\le c\left(e^{-\frac{c'}{t^{\eta'}}}\right).
\]
On the other hand, if we denote the Brownian meander of length $\frac{1}{\tau}$ by $m^{1/\tau}_s$, namely,
\[
m^{1/\tau}_s=\frac{1}{\sqrt{\tau}}M_{s\tau}^N, \quad s\in [0,1/\tau],
\]
then we know that $m^{1/\tau}_{1/\tau}$ is Rayleigh distributed with scale parameter $\frac{1}{\sqrt{\tau}}$, and hence
\[
\p\bigg(\left|m^{1/\tau}_{1/\tau}\right|_*>R\bigg)\le 2e^{-\frac{R^2\tau}{2}}
\]
for any $R>0$. This implies that for any $\tau\in[0,1]$
\[
\p\bigg(\left|m^{1/\tau}_{1/\tau}\right|_*>\frac{1}{u^\eta\tau}\bigg)\le 2e^{-\frac{1}{2\tau u^{2\eta}}}.
\]
Hence we have
\begin{align*}
&\p\bigg(\exists \tau\in[0,1],\,{u}^{\eta}\left(\int_0^\tau (M^N_s)^2ds\right)^{1/2} >\sqrt{\tau}\bigg) \\
& \quad =\p\bigg(\exists \tau\in[0,1],\,{u}^{\eta}\left(\tau^2\int_0^1 (m^{1/\tau}_s)^2ds\right)^{1/2} >\sqrt{\tau}\bigg)\\
& \qquad =\p\bigg(\exists \tau\in[0,1],\,{u}^{2\eta}\int_0^1 (m^{1/\tau}_s)^2ds>\frac{1}{\tau}\bigg) \\
& \qquad \quad \le \p\bigg(\exists \tau\in[0,1],\,{u}^{2\eta}\left|m^{1/\tau}_{1/\tau}\right|_*>\frac{1}{\tau}\bigg)
\end{align*}
and moreover
\begin{align*}
&\p\bigg(\exists \tau\in[0,1],\,{u}^{2\eta}\left|m^{1/\tau}_{1/\tau}\right|_*>\frac{1}{\tau}\bigg) \\
& \quad \le\sum_{k=0}^\infty\p\bigg(\exists \tau\in(2^{-(k+1)},2^{-k}],\,\left|m^{1/\tau}_{1/\tau}\right|_*>\frac{1}{\tau {u}^{2\eta}}\bigg)\\
& \qquad \le\sum_{k=0}^\infty \p\bigg(\exists \tau\in(2^{-(k+1)},2^{-k}],\,\left|m^{1/\tau}_{1/\tau}\right|_*>\frac{2^k}{ {u}^{2\eta}}\bigg) \le\sum_{k=0}^\infty 2e^{-\frac{4^k}{ 2^{k+1}u^{4\eta}}}= o(e^{-\frac{c}{u^{4\eta}}})
\end{align*}
for some $c>0$. Hence we complete the proof.
\end{proof}

Let $H=\max\{\max_{s\in\pO}(H_{\pO,0}(s)),0\}$ and $K'_1=\max\{\max_{s\in\pO}k_1(s),0\}$. 
From Lemma \ref{lemma-MN-MT}, for any small $\eta>0$, in the set $\{|M^T_\tau|_*<\delta u^{-1/4}\}\cap\{u^{2\eta}|M^N_\tau|_*\le1 \}$, there exist $c, \eta'>0$ such that
\begin{align*}
&\frac{8}{\delta^2}{u}\bigg|\bigg(M^T_{\tau},-(M^N_\tau y-M^T_\tau \xi)-\sqrt{u}M^N_\tau M^T_\tau+2\sqrt{u}\int_0^\tau M_s^NdM_s^T\bigg)\bigg|^3\\
&\quad\le
\frac{C}{\delta^2}u\bigg(|M^T_\tau|^3+|M^T_\tau \xi|^{3/2}+(M^N_\tau)^{3/2} |y|^{3/2}+u^{3/4-3\eta}|M_\tau^T|_*^{3/2}\bigg)\\
&\qquad \le \frac{C}{\delta^2}u |y|^{3/2}(M^N_\tau)^{3/2}+{C'}u^{3/4}|M^T_\tau|_*^2+C''\left(u^{7/8}|\xi|^{3/2}+u^{5/8-3\eta}\right)|M^T_\tau|_*\\
&\qquad \quad \le  C_\delta\bigg(u^{1-\eta} |y|^{3/2}M^N_\tau+u^{3/4}|M^T_\tau|_*^2+\left(u^{7/8}|\xi|^{3/2}+u^{5/8-3\eta}\right)|M^T_\tau|_*\bigg)
\end{align*}
where $C', C'', C_\delta>0$ are constants depend only on $\delta$. Then we have
\begin{align*}
W(y,u,z,\tau) &\le  \bigg(H|y|+K'_1|\xi|+3K'_1u^{1/2-2\eta}+C_\delta \left(u^{7/8}|\xi|^{3/2}+u^{5/8-3\eta}\right)\bigg)|M^T_{\tau}|_*\\
& \quad +\left(\frac{1}{2}H\sqrt{u}+C_\delta u^{3/4}\right)|M^T_{\tau}|_*^2+(-k_1(s)y+C_\delta u^{1-\eta} |y|^{3/2} ) M^N_\tau.
\end{align*}
We denote
$$
a=H|y|+K'_1|\xi|+3K'_1u^{1/2-2\eta}+C_\delta \left(u^{7/8}|\xi|^{3/2}+u^{5/8-3\eta}\right)
$$
and
$$
b=\frac{1}{2}H\sqrt{u}+C_\delta u^{3/4}.
$$
Hence in the set  $\{|M^T_\tau|_*<\delta u^{-1/4}\}\cap\{u^{2\eta}|M^N_\tau|_*\le1 \}\cap \{|-k_1(s)y+C_\delta u^{1-\eta} |y|^{3/2}|<1/2 \}$ 
it holds that
\[
W(y,u,z,\tau)\le a|M^T_{\tau}|_*+ b|M^T_{\tau}|_*^2+\frac{1}{2}M^N_\tau.
\]%
Let $\gamma=u^{-1/2}\left(\xi-r+\frac{1}{2}H_{\pO,0}(s)|y|^2+k_1(s)z+\frac{8}{\delta^2} |(y,z)|^3\right)$, then  from \eqref{eq-D-W} we have 
we have
\begin{align}\label{eq-int-psr}
 D(s,r,t) &\le \int_0^\ep dr\int_0^\infty d\xi\int_{-\infty}^{\infty}dy\int_{-\infty}^{\infty}dz  \int_0^t du\, \chi_t(\xi, y, z, u){\mathbbm{1}}_{\xi-r+\frac{1}{2}H_{\pO,0}(s)|y|^2+k_1(s)z<-\frac{1}{\delta^2}  |(y,z)|^3} \nonumber  \\
&\quad \cdot\bigg[\p_x\bigg( \exists \tau\in [0,1]: M^N_{\tau}\le 2\gamma +2a|M^T_\tau|_*+2b|M^T_\tau|_*^2 \bigg)+\p_x\bigg( \exists \tau\in [0,1]: |M^T_\tau|_*\ge\delta u^{-1/4} \bigg)  \nonumber\\
&\qquad+\p_x\bigg( \exists \tau\in [0,1]:  u^{2\eta}|M^N_\tau|_*>1\bigg) +\p_x\bigg( \exists \tau\in [0,1]: K_1'|M^T_t|_* +C_\delta t^{1-\eta}|M^T_t|_*^{3/2}\ge1/2\bigg)\bigg].
\end{align}
At last since
\begin{align*}
&\int_0^\ep dr\int_0^\infty d\xi\int_{-\infty}^{\infty}dy \int_{-\infty}^{\infty}dz \int_0^t du\, \chi_t(\xi, y, z, u){\mathbbm{1}}_{-\frac{8}{\delta^2} |(y,z)|^3\le\xi-r+\frac{1}{2}H_{\pO,0}(s)|y|^2+k_1(s)z<\frac{1}{\delta^2}|(y,z)|^3} \\
&\quad \le \int_0^\infty d\xi\int_{-\infty}^{\infty}dy \int_{-\infty}^{\infty}dz  \int_0^t du\, \chi_t(\xi, y, u)\frac{9}{\delta^2} |(y,z)|^3\\
&\qquad \le\frac{72}{\delta^2} \E\left(|B^T_{\tau_t}|^3+|A_{\tau_t}|^{3/2} \right)=\frac{72}{\delta^2}\E(|B_{\tau_1}|^3+|A_{\tau_1}|^{3/2})t^{3/2}=o(t^{3/2}).
\end{align*}
We obtain that
\[
\int_0^\ep D(s,r,t)dr\le A_1(s,t)+A_2(s,t)+A_3(s,t)+A_4(s,t)+o(t^{3/2}),
\]
where
\begin{align*}
A_1(s,t)
 &= \int_0^\ep dr\int_0^\infty d\xi\int_{-\infty}^{\infty}dy\int_{-\infty}^{\infty}dz  \int_0^t du\, \chi_t(\xi, y, z, u){\mathbbm{1}}_{\xi-r+\frac{1}{2}H_{\pO,0}(s)|y|^2+k_1(s)z<-\frac{8}{\delta^2} |(y,z)|^3}   \\
& \cdot\p_{(s,r)}\bigg( \exists \tau\in [0,1]: M^N_{\tau}\le 2\gamma +2a|M^T_\tau|_*+2b|M^T_\tau|_*^2 \bigg)\,,
 \end{align*}
 \begin{align*}
A_2(s,t)
 &= \int_0^\ep dr\int_0^\infty d\xi\int_{-\infty}^{\infty}dy\int_{-\infty}^{\infty}dz  \int_0^t du\, \chi_t(\xi, y, z, u){\mathbbm{1}}_{\xi-r+\frac{1}{2}H_{\pO,0}(s)|y|^2+k_1(s)z<-\frac{8}{\delta^2} |(y,z)|^3}   \\
& \cdot\p_{(s,r)}\bigg( \exists \tau\in [0,1]: |M^T_\tau|_*\ge\delta u^{-1/4} \bigg)\\
&\le \p\bigg(|M^T_1|_*\ge\delta t^{-1/4}\bigg)\le C'e^{-\frac{C}{t^{1/2}}}\,,
 \end{align*}
 \begin{align*}
A_3(s,t)
 &= \int_0^\ep dr\int_0^\infty d\xi\int_{-\infty}^{\infty}dy\int_{-\infty}^{\infty}dz  \int_0^t du\, \chi_t(\xi, y, z, u){\mathbbm{1}}_{\xi-r+\frac{1}{2}H_{\pO,0}(s)|y|^2+k_1(s)z<-\frac{8}{\delta^2} |(y,z)|^3}   \\
& \cdot\p_{(s,r)}\bigg( \exists \tau\in [0,1]: u^{2\eta}|M^N_\tau|_*>1 \bigg)\\
&\le \p\bigg( \exists \tau\in [0,1]: t^{2\eta}|M^N_\tau|_*>1 \bigg)=\p\bigg(|M^N_1|_*> t^{-2\eta} \bigg)\le 2e^{-\frac{1}{2t^{4\eta}}}\,,
 \end{align*}
  and
 \begin{align*}
A_4(s,t)
 &= \int_0^\ep dr\int_0^\infty d\xi\int_{-\infty}^{\infty}dy\int_{-\infty}^{\infty}dz  \int_0^t du\, \chi_t(\xi, y, z, u){\mathbbm{1}}_{\xi-r+\frac{1}{2}H_{\pO,0}(s)|y|^2+k_1(s)z<-\frac{8}{\delta^2} |(y,z)|^3}   \\
& {\mathbbm{1}}_{ |k_1(s)y+C_\delta u^{1-\eta} |y|^{3/2}|\ge1/2}\\
&\le \p\bigg( K_1'|M^T_t|_* +C_\delta t^{1-\eta}|M^T_t|_*^{3/2}\ge1/2 \bigg)\le \p\bigg( |M^T_t|_*\ge C\bigg)=O(e^{-c/t}).
 \end{align*}
We are left to show that $A_1(t)=O(t^{3/2-\eta})$ for some small $\eta>0$.
\begin{lemma}\label{lemma-meander}
There exist constants $C>0$, $C'>0$ that are independent of $a, b, \gamma$ such that
\[
\p\bigg( \exists \tau\in [0,1]: M^N_{\tau}\le 2\gamma +2a|M^T_\tau|_*+2b|M^T_\tau|_*^2 \bigg)\le C(a+b)(1+(\log^+(1/x_0))e^{-C'x_0^2},
\]
where $(\cdot)^+$ denotes the positive part and $x_0=\frac{-a+\sqrt{a^2-4\gamma b}}{2b}$.
\end{lemma}
\begin{proof}
The proof can be found in \cite{vdbg:manifold}, Lemma 4.3 by replacing a $d$-dimensional Bessel process with $S_\tau=\sup_{0\le u\le \tau}|M^T_u|$ where $M^T_u$ is a standard Brownian motion and noting that
\[
\p\bigg(\sup_{0\le \tau\le 1}S_\tau\ge \xi \bigg)=\p\bigg(\sup_{0\le u\le 1}|M^T_u|\ge \xi \bigg)\le Ce^{-\xi^2/8\tau}.
\]
\end{proof}
From the above lemma we have
 \begin{align*}
 A_1(t)
 &\le \int_0^\ep dr\int_0^\infty d\xi\int_{-\infty}^{\infty}dy\int_{-\infty}^{\infty}dz  \int_0^t du\, \chi_t(\xi, y, z, u){\mathbbm{1}}_{\xi-r+\frac{1}{2}H_{\pO,0}(s)|y|^2+k_1(s)z<-\frac{8}{\delta^2}|(y,z)|^3}   \\
& \cdot C(a+b)(a+\log^+(1/x_0))e^{-C'x_0^2}.
 \end{align*}
By the change of variables $r\to\gamma$ we obtain
 \begin{align*}
A_1(t)
 &\le \sqrt{t}\int_{-\infty}^0 d\gamma\int_0^\infty d\xi\int_{-\infty}^{\infty}dy\int_{-\infty}^{\infty}dz  \int_0^t du\, \chi_t(\xi, y, z, u)   C(a+b)(1+\log^+(1/x_0))e^{-C'x_0^2}.
 \end{align*}
 Since $\gamma=\frac{a^2-(a+2bx_0)^2}{4b}$, by the change of variables $\gamma\to x_0$ we have
  \begin{align*}
A_1(t) &\le C\sqrt{t}\int_0^\infty d\xi\int_{-\infty}^{\infty}dy\int_{-\infty}^{\infty}dz  \int_0^t du\, \chi_t(\xi, y, z, u)\\
 &\cdot  \int_0^\infty dx_0 (a+b)(a+2b x_0)(1+\log^+(1/x_0))e^{-C'x_0^2}\\
 &=C\sqrt{t}\int_0^\infty d\xi\int_{-\infty}^{\infty}dy\int_{-\infty}^{\infty}dz  \int_0^t du\, \chi_t(\xi, y, z, u) (C_1a(a+b)+2C_2b(a+b))
 \end{align*}
 where  $C_1= \int_0^\infty (1+\log^+(1/x))e^{-C'x^2}dx$ and  $C_2= \int_0^\infty x(1+\log^+(1/x))e^{-C'x^2}dx$ are positive constants.
Now since  for some $C, C'>0$ and $\eta>0$ small enough we have
 \begin{align*}
&\int_0^\infty d\xi\int_{-\infty}^{\infty}dy\int_{-\infty}^{\infty}dz  \int_0^t du\, \chi_t(\xi, y, z, u) a^2\\
&\le 2^5\left( H^2\E(|B^T_{\tau_t}|^2)+{K'_1}^2\E(|B^N_{\tau_t}|^2)+9{K'_1}^2t^{1-4\eta}+C_\delta^2t^{7/4}E(|B^N_{\tau_t}|^{3/2})+C_\delta^2t^{5/4-6\eta}\right)=Ct^{1-4\eta}+O(t),
 \end{align*}
 \begin{align*}
&\int_0^\infty d\xi\int_{-\infty}^{\infty}dy\int_{-\infty}^{\infty}dz  \int_0^t du\, \chi_t(\xi, y, z, u) ab\\
& \le
\frac{1}{2}H^2\E(|B^T_{\tau_t}|)(\sqrt{t}+o(t))+\frac{1}{2}HK'_1\E(|B^N_{\tau_t}|)(\sqrt{t}+o(t))+\frac{3}{2}{K'_1}Ht^{1-\eta}+o(t)=C'\,t^{1-2\eta}+O(t)
\end{align*}
and
\begin{align*}
&\int_0^\infty d\xi\int_{-\infty}^{\infty}dy\int_{-\infty}^{\infty}dz  \int_0^t du\, \chi_t(\xi, y, z, u) b^2\le
\frac{1}{4}H^2t+ o(t),
\end{align*}
we finally obtain that $A_1(s,t)=O(t^{3/2-\eta'})$ for $\eta'>0$ as small as we want. At the end from \eqref{eq-int-p} and \eqref{eq-P-D} we have
\begin{align*}
\int_{\Omega_\ep} \p_x(\tau_t<T'_\Omega\le t)dx&\le(1+K'\ep)\int_{\pO}\left(\int_0^\ep D(s,r,t)dr\right)d\sigma_0(s)+O(t^{3/2})\\
&\le(1+K'\ep)\,\sigma_0(\pO)\cdot O(t^{3/2-\eta'})
\end{align*}
Thus we complete the proof of Lemma \ref{lemma-main-III}.
\end{proof}
\appendix
\section{Appendix}\label{sec:appendix}

\subsection{Proof of equation \eqref{eq-claim-A5}}\label{App-claim-A5}
We want to show
\[
\int_0^\ep\int_{\partial \Omega} \p_x(B^N_{\tau_t}<r,x'_{\tau_t}\not\in \Omega, |(B^T_{\tau_t},A_{\tau_t})|<\delta)(1-J_\Psi(s,r))d\sigma_0(s)dr=O(t^{3/2}).
\]
By Lemma \ref{lemma-J} we know that for small enough $\ep>\delta>0$ there exists a $K_2>0$ such that $1-J_\Psi(s,r)\le K_2r$ for all $s\in\pO$. It then suffices to show that
\[
\int_0^\ep\int_{\partial \Omega} \p_x(B^N_{\tau_t}<r,x'_{\tau_t}\not\in \Omega, |(B^T_{\tau_t},A_{\tau_t})|<\delta)rd\sigma_0(s)dr=O(t^{3/2}).
\]
By \eqref{eq-NTZ-cor} we have $\{x'_{\tau_t}\not\in \Omega\}=\{h(B^T_{\tau_t}, A_{\tau_t};s)>r-B^N_{\tau_t}\}$ in $\mathcal{O}_x$. Thus we just need to show
\[
\int_0^\infty \p_x(B^N_{\tau_t}<r<B^N_{\tau_t}+h(B^T_{\tau_t}, A_{\tau_t};s), |(B^T_{\tau_t},A_{\tau_t})|<\delta)rdr=O(t^{3/2}).
\]
By changing order of integrals we have the left hand side of the above equation given by
\begin{align*}
&\E_x\left(\int_0^\infty {\mathbbm{1}}_{\{B^N_{\tau_t}<r<B^N_{\tau_t}+h(B^T_{\tau_t}, A_{\tau_t};s) \}}{\mathbbm{1}}_{\{ |(B^T_{\tau_t},A_{\tau_t})|<\delta\}} rdr\right)\\
 &=\frac{1}{2} \E_x\left[\bigg(\left(h^+(B^T_{\tau_t}, A_{\tau_t};s)\right)^2+2h^+(B^T_{\tau_t}, A_{\tau_t};s)B^N_{\tau_t}\bigg) {\mathbbm{1}}_{\{ |(B^T_{\tau_t},A_{\tau_t})|<\delta\}}   \right].
\end{align*}
Here $h^+$ is the positive part of $h$. By Lemma \ref{lemma-h-H} we know that there exists a constant $C_{\delta, s}>0$ depending on $\delta>0$ and $s\in \pO$ such that $|h(y, z;p)|\le C_{\delta, s}(y^2/2+|z|)$. Hence
\begin{align*}
&\E_x\left[\bigg(\left(h^+(B^T_{\tau_t}, A_{\tau_t};s)\right)^2+2h^+(B^T_{\tau_t}, A_{\tau_t};s)B^N_{\tau_t}\bigg) {\mathbbm{1}}_{\{ |(B^T_{\tau_t},A_{\tau_t})|<\delta\}}   \right]\\
&\le
C_1 \E_x\left[(B^T_{\tau_t})^4 \right]+C_2 \E_x\left[|A_{\tau_t}|^2 \right]
+C_3 \E_x\left[(B^T_{\tau_t})^2B^N_{\tau_t} \right]+C_4 \E_x\left[|A_{\tau_t}|\cdot B^N_{\tau_t} \right]
\end{align*}
for some positive constants $C_1, \dots, C_4$ depending on $\delta>0$ and $s\in\pO$.
We are left to show each of the five terms above is mostly of scale $O(t^{3/2})$. This can be checked easily as below.
\begin{align*}
 \E_x\left((B^T_{\tau_t})^4 \right)\le \E_x\left(|B^T_{t}|_*^4 \right)=t^2 \E_x\left(|B^T_{1}|_*^4 \right)=O(t^2),
\end{align*}
where $|B^T_{t}|_*$ is the running maximum of $|B^T_{t}|$.
Next, note $A_s=\beta_{\int_0^s (B^T_u)^2+(B^N_u)^2du}$ for all $s>0$ where $\beta$ is an independent standard Brownian motion, we have
\begin{align*}
 &\E_x\left[|A_{\tau_t}|^2 \right]= \E_x\left[ \int_0^{\tau_t} (B^T_u)^2+(B^N_u)^2du\right]\E_x\left(|\beta_1|^2\right) \\
 &\le t \E_x\left(|B^T_t|_*^2\right)+t \E_x\left(|B^N_t|_*^2\right) =O(t^2).
\end{align*}
Also by independence we have
\begin{align*}
 &\E_x\left((B^T_{\tau_t})^2B^N_{\tau_t} \right)=\E_x\left((B^T_{\tau_t})^2\right)\,\E_x\left(B^N_{\tau_t} \right)\le \E_x\left(|B^T_{t}|_*^2\right)\,\E_x\left(|B^N_{t}|_* \right)=O(t^{3/2}).
 \end{align*}
 At last since
  \begin{align*}
 &\E_x\left[|A_{\tau_t}|\cdot B^N_{\tau_t} \right]= \E_x\left[ \bigg(\int_0^{\tau_t} (B^T_u)^2+(B^N_u)^2du\bigg)^{1/2} B^N_{\tau_t}\right]\E_x\left(|\beta_1|\right) \\
 &\le\bigg( t \E_x\left(|B^T_t|_*^2\right)+t \E_x\left(|B^N_t|_*^2\right) \bigg)  \E_x\left( |B^N_{t}|\right)=O(t^{3/2}),
  \end{align*}
we complete the proof.

\subsection{Proof of equation \eqref{eq-claim-A3}}\label{App-claim-A3}
Note that for any $\ep>0$,
\begin{align*}
&\int_\ep^\infty\int_{\partial \Omega} \p_x(B^N_{\tau_t}<r,x'_{\tau_t}\not\in \Omega, |(B^T_{\tau_t}, A_{\tau_t})|<\delta)d\sigma_0(s)dr\\
&\le
\int_0^\infty\int_{\partial \Omega} \p_x(B^N_{\tau_t}<r,x'_{\tau_t}\not\in \Omega, |(B^T_{\tau_t}, A_{\tau_t})|<\delta)d\sigma_0(s)\frac{r}{\ep}dr.
\end{align*}
We then prove the claim by using the argument in above section.

\subsection{Proof of equation \eqref{eq-minor-I3-1}}\label{App-minor-I3-1}
The proof is similar to that of \eqref{eq-claim-A3}. We just need to show
\[
\int_0^\infty \p_x(B^N_{\tau_t}+h(B^T_{\tau_t}, A_{\tau_t};s)<r<B^N_{\tau_t}, |(B^T_{\tau_t},A_{\tau_t})|<\delta)rdr=O(t^{3/2}),
\]
that is to show
\begin{align*}
 &\E_x\left(\int_0^\infty {\mathbbm{1}}_{\{B^N_{\tau_t}<r<B^N_{\tau_t}+h(B^T_{\tau_t}, A_{\tau_t};s) \}}{\mathbbm{1}}_{\{ |(B^T_{\tau_t},A_{\tau_t})|<\delta\}} rdr\right)\\
 &=\frac{1}{2} \E_x\left[\bigg(\left(h^-(B^T_{\tau_t}, A_{\tau_t};s)\right)^2+2h^-(B^T_{\tau_t}, A_{\tau_t};s)B^N_{\tau_t}\bigg) {\mathbbm{1}}_{\{ |(B^T_{\tau_t},A_{\tau_t})|<\delta\}}   \right]=O(t^{3/2}),
\end{align*}
where $h^-(y,z;s)$ is the negative part of $h(y,z;s)$. The proof is as in Section \ref{App-claim-A3}.

\subsection{Proof of equation \eqref{eq-minor-I3-2}}\label{App-minor-I3-2}
Since
\begin{align*}
&\int_\ep^\infty\int_{\partial \Omega} \p_x(B^N_{\tau_t}>r,x'_{\tau_t}\in \Omega, |(B^T_{\tau_t},A_{\tau_t})|<\delta)d\sigma_0(s)dr\\
&\le\int_0^\infty \int_{\partial \Omega} \p_x(B^N_{\tau_t}>r,x'_{\tau_t}\in \Omega, |(B^T_{\tau_t},A_{\tau_t})|<\delta)\,d\sigma_0(s)\frac{r}{\ep}dr
\end{align*}
By using the same argument as in Section \ref{App-minor-I3-1} we conclude that the above term is $O(t^{3/2})$.

\subsection{Proof of equation \eqref{eq-C-4}}\label{App-C-4}
\begin{lemma}\label{lemma-A-tau}
Let $(-B^N_t, B^T_t, A_t)$ and $\tau_t$ be defined as before. Then for any small $\eta>0$, we have some $C>0$ such that
$\p(B^T_{\tau_t}>\tau_t^{1/2-\eta/2})=O(e^{-\frac{C}{t^\eta}})$ and $\p(|A_{\tau_t}|>|\tau_t|^{1-\eta})=O(e^{-\frac{C}{t^\eta}})$.
\end{lemma}

\begin{proof}
First estimate comes from the independence of $B^T$ and $(B^N, \tau_t)$, hence
 \[
\p(B^T_{\tau_t}>\tau_t^{1/2-\eta/2})\le\p(|B^T_{1}|_*>\tau_t^{-\eta/2})\le \p(|B^T_{1}|_*>t^{-\eta/2})=O(e^{-\frac{C}{t^\eta}}).
 \]
To show the second estimate, from  \eqref{eq-phi} we have
 \begin{align*}
 &\p(B^N_{\tau_t}>\tau_t^{1/2-\eta/2})=\int_0^t\int_{\tau^{1/2-\eta/2}}^\infty\Phi(\xi,\tau;t)d\xi d\tau=\int_0^t\frac{e^{-\frac{1}{2\tau^{\eta}}}}{\pi\tau^{1/2}\sqrt{t-\tau}}d\tau\\
 &\le e^{-\frac{1}{2t^\eta}}\int_0^t\frac{1}{\pi\tau^{1/2}\sqrt{t-\tau}}d\tau= e^{-\frac{1}{2t^\eta}}.
 \end{align*}
 By considering $A_s$ as a time changed Brownian motion, i.e., $A_s=\beta_{\int_0^s (B^T_u)^2+(B^N_u)^2du}$ for all $s>0$ where $\beta$ is an independent standard Brownian motion, we obtain
  \begin{align*}
  &\p(|A_{\tau_t}|>|\tau_t|^{1-\eta})=\p\bigg(|\beta_{1}|>\frac{|\tau_t|^{1-\eta}}{\sqrt{\int_0^{\tau_t}(B^T_u)^2+(B^N_u)^2du}}\bigg)
\le \p\bigg(|\beta_{1}|>\frac{|\tau_t|^{1-\eta}}{|B^T_{\tau_t}|_*\sqrt{\tau_t}+B^N_{\tau_t}\sqrt{\tau_t}}\bigg)\\
  &\le  \p\bigg(|\beta_{1}|>\frac{|\tau_t|^{1-\eta}}{2\tau_t^{1-\eta/2}}\bigg)+\p (|B^T_{\tau_t}|_*>\tau_t^{1/2-\eta/2}) +\p(B^N_{\tau_t}>\tau_t^{1/2-\eta/2})\\
  &= \p\bigg(|\beta_{1}|>\frac{1}{2\tau_t^{\eta/2}}\bigg)+O(e^{-\frac{C}{t^\eta}})=O(e^{-\frac{C}{t^\eta}}).
   \end{align*}
   Hence the conclusion.
 \end{proof}
Now we are ready to show $|R_3(t)|=O(t^{3/2})$. It suffices to show
\[
\int_{\partial \Omega}  \,\E_x\left(\max\left(h^-(B^T_{\tau_t},A_{\tau_t};s)-B^N_{\tau_t},0\right){\mathbbm{1}}_{\{|(B^T_{\tau_t},A_{\tau_t})|<\delta\} }\right)\,d\sigma_0(s)=O(t^{3/2}).
\]
Note for $\xi\in[0,h^-(y,z;s))$, it holds that $|\xi-h^-(y,z;s)|\le h^-(y,z;s)$, hence we have
\[
|R_3(t)|\le \int_{\partial \Omega}  \,\E_x\left(h^-(B^T_{\tau_t},A_{\tau_t};s){\mathbbm{1}}_{\{B^N_{\tau_t}<h^-(B^T_{\tau_t},A_{\tau_t};s)\}}{\mathbbm{1}}_{\{|(B^T_{\tau_t},A_{\tau_t})|<\delta\} }\right)d\sigma_0(s)
\]
By Lemma \ref{lemma-h-H} we know that there exists $C_{\delta, s}>0$ depending on $\delta>0$ and $s\in \pO$ such that $|h(y, z;p)|\le C_{\delta, s}(y^2/2+|z|)$. By compactness of $\Omega$ we have
\begin{align*}
&|R_3(t)|\le C\, \E_x\left(\left(|B^T_{\tau_t}|^2+|A_{\tau_t}|\right){\mathbbm{1}}_{\{B^N_{\tau_t}<|B^T_{\tau_t}|^2+|A_{\tau_t}|\}}\right)\\
&\le
C\, \E_x\left(\left(|B^T_{\tau_t}|^2+|\tau_t|^{1-\eta}\right){\mathbbm{1}}_{\{B^N_{\tau_t}<2|\tau_t|^{1-\eta}\}}\right)+O(e^{-\frac{C}{t^\eta}}).
\end{align*}
for some constant $C>0$. The last inequality comes from Lemma \ref{lemma-A-tau}. Using \eqref{eq-phi} we have
\begin{align*}
&\E_x\bigg(|B^T_{\tau_t}|^2\cdot \p\left(B^N_{\tau_t}<2|\tau_t|^{1-\eta}\right)\bigg)=\E_x\bigg(|B^T_{\tau_t}|^2\bigg)\cdot \int_0^t\int_0^{2\tau^{1-\eta}}\Phi(\xi,\tau;t)d\xi d\tau\\
&\le t\,\int_0^t\frac{1-e^{-2\tau^{1-2\eta}}}{\pi\sqrt{\tau}\sqrt{t-\tau}}d\tau\le t\cdot C\,t^{1-2\eta}=O(t^{3/2}).
\end{align*}
Moreover, since
\begin{align*}
&\E_x\bigg(|\tau_t|^{1-\eta}\cdot \p\left(B^N_{\tau_t}<2|\tau_t|^{1-\eta}\right)\bigg)=
\int_0^t\int_0^{2\tau^{1-\eta}}\tau^{1-\eta}\,\Phi(\xi,\tau;t)d\xi d\tau\le\int_0^t\frac{2\tau^{2-3\eta}}{\pi\sqrt{\tau}\sqrt{t-\tau}}d\tau=O(t^{3/2}).
\end{align*}
We can then conclude that $|R_3(t)|=O(t^{3/2})$.

\subsection{Proof of \eqref{eq-C}}\label{App-C}
From Lemma \ref{lemma-h-H} and Lemma  \ref{lemma-A-tau} we obtain for some $C>0$ and any small $\eta>0$,
\begin{align*}
|C_1(t)|
&\le C\,\E_x\left( (|B^T_{\tau_t}|^3+|B^T_{\tau_t}A_{\tau_t}|){\mathbbm{1}}_{\{|(B^T_{\tau_t},A_{\tau_t})|<\delta\} }\right) \le
 C\,\E_x\bigg(2|\tau_t|^{3/2-3\eta/2} \bigg)+O(e^{-\frac{C}{t^\eta}})\\
&=\int_0^t \frac{\tau^{3/2-3\eta/2}}{\pi\tau^{1/2}(t-\tau)^{1/2}}
 =O(t^{3/2-3\eta/2})
\end{align*}
Similarly we have for some $C>0$ it holds that
\begin{align*}
C_2(t)\le
&\frac{1}{\delta}\left(\frac{1}{2}H+K_1\right)\int_{\partial \Omega}  \,\E_x\left(\left(|B^T_{\tau_t}|^2+|A_{\tau_t}|\right)^{3/2}\mathbbm{1}_{\{|(B^T_{\tau_t},A_{\tau_t})|\ge\delta\} }\right)\,d\sigma_0(s)\\
&\le \frac{C}{\delta}  \,\E\bigg( (|B^T_{\tau_t}|^3+|A_{\tau_t}|^{3/2})\bigg)=O(t^{3/2-3\eta/2}),
\end{align*}
where $H=\max_{s\in \pO}\{H_{\pO,0}(s) \}$ and $K_1=\max_{s\in \pO}\{k_1(s) \}$.


\end{document}